\documentclass[oneside,reqno]{amsart}

\usepackage{amsfonts}
\usepackage{amsmath}
\usepackage{amssymb}
\usepackage[english]{babel}
\usepackage{amsthm}
\usepackage{graphicx}
\usepackage{color}

\newtheorem {thm}{Theorem}[section]
\newtheorem {lem}[thm]{Lemma}

\newtheorem {cor}[thm]{Corollary}
\newtheorem{conj}[thm]{Conjecture}

\theoremstyle{definition}
\newtheorem {defn}[thm]{Definition}

\theoremstyle{remark}
\newtheorem*{rem*}{Remark}

\DeclareMathOperator{\aff}{aff}
\DeclareMathOperator{\lin}{lin}
\DeclareMathOperator{\proj}{proj}

\DeclareMathOperator{\fan}{Fan}
\DeclareMathOperator{\inter}{rel\,int}
\DeclareMathOperator{\conv}{conv}
\newcommand{\pcap}{\mathrm{Cap}}

\newcommand{\belt}{\mathrm{Belt}}

\title{Voronoi's Conjecture for extensions of Voronoi parallelohedra}

\author{Alexander Magazinov}
\address{Steklov Mathematical Institute of the Russian Academy of Sciences, 8 Gubkina street, Moscow 119991, Russia \\
Yaroslavl State University, 14 Sovetskaya street, Yaroslavl 150000, Russia.}
\email{magazinov-al@yandex.ru}

\begin{document}
\sloppy

\begin{abstract}

Let $I$ be a segment in the $d$-dimensional Euclidean space $\mathbb E^d$. Let $P$ and $P+I$ be parallelohedra in $\mathbb E^d$, where ``+''
denotes the Minkowski sum. We prove that Voronoi's Conjecture holds for $P+I$, i.e. $P+I$ is a Voronoi parallelohedron for some 
Euclidean metric in $\mathbb E^d$, if Voronoi's Conjecture holds for $P$.
\end{abstract}

\maketitle

\section{Introduction}

This paper focuses on {\it parallelohedra}, which are, by definition, convex polytopes that tile Euclidean space in a face-to-face way. The notion of
parallelohedra was introduced by E.~S.~Fedorov~\cite{Fed1885} in 1885.

\subsection{Notation and general properties}\hspace*{\fill}

In the paper we will use the notation $\lin L$ for the linear space associated with the affine space $L$. If $\mathcal S$ is a set of
points, then $\aff \mathcal S$ denotes the affine hull of $\mathcal S$, i.e. the minimal affine space containing all its points. If some
elements of $\mathcal S$ are vectors and some are points, then the vectors are identified with the endpoints of the equal radius-vectors. Then $\aff \mathcal S$
is the affine hull of the resulting point set. 

The particularly important usage of the notation is $\lin \aff \mathcal S$ when $\mathcal S$ is a
set of vectors. One can check that $\lin \aff \mathcal S$ is the space of all linear combinations of vectors of $\mathcal S$ with sum of coefficients
equal to 0.

We will also need the notation for linear projections. In this paper $\proj_p$ denotes the projection {\it along} the linear subspace $p$ onto some
complementary affine subspace $p'$. If needed, $p'$ is specified separately, otherwise it is chosen arbitrarily. 

For the sake of brevity, we will also write $\proj_M$ instead of $\proj_{\lin \aff M}$ when $M$ is a polytope (in this paper the common cases are 
that $M$ is a segment or a face of a parallelohedron).

We recall some general properties of parallelohedra. 

\begin{enumerate}
	\item[1.] If $T(P)$ is a face-to-face tiling of $\mathbb E^d$ by translates of $P$, then
	$$\Lambda(P) = \{ \mathbf t: P+\mathbf t \in T(P) \}$$
	is a $d$-dimensional lattice.
	\item[2.] $P$ has a center of symmetry.
	\item[3.] Each facet of $P$ (i.e. a $(d-1)$-dimensional face of $P$) has a center of symmetry.
\end{enumerate}

\begin{defn}
Consider an arbitrary face $F\subset P$ of dimension $d-2$. The set of all facets of $P$ parallel to $F$ is called
the {\it belt} of $P$ determined by $F$ and denoted by $\belt(F)$. Each facet of $\belt(F)$ contains two $(d-2)$-faces parallel and congruent to $F$, and
each $(d-2)$-dimensional face of $P$ parallel to $F$ is shared by two facets of $\belt(F)$.
\end{defn}

\begin{enumerate}	
	\item[4.] For every $(d-2)$-dimensional face $F\subset P$ the belt $\belt(F)$ consists of 4 or 6 facets. It means that $\proj_{F}(P)$
	is a parallelogram or a centrally symmetric hexagon.
\end{enumerate}

Properties 1 -- 4 were established in~\cite{Min1897}. B.~Venkov~\cite{Ven1954} proved that every convex polytope satisfying conditions 2, 3 and 4
({\it Minkowski--Venkov conditions}) is a parallelohedron.

\begin{defn}
For simplicity, assume $\mathbf 0 \in \Lambda(P)$, i.e. $P\in T(P)$. A {\it standard face} of $P$ is a face that can be represented as $P\cap P'$, 
where $P'\in T(P)$.
\end{defn}

\begin{enumerate}	
	\item[5.] Let $F = P \cap P'$ be a standard face of $P$, where $P' = P + \mathbf t$, $\mathbf t\in \Lambda(P)$. Then the point
	$\mathbf t/2$ is the center of symmetry of $F$. The vector $\mathbf t$ will be called {\it the standard vector} of $F$ and denoted by
	$\mathbf s(F)$.
\end{enumerate}

The central symmetry of such faces is established in \cite{Hor1999} by \'A.~Horv\'ath.
N.~Dolbilin introduced the term ``standard face''. His paper~\cite{Dol2009} establishes several useful properties of standard faces.

There are some important particular cases of standard faces. If $F$ is a facet of $P$, then $F$ is necessarily standard. Then the notion of a standard
vector $\mathbf s(F)$ coincides with the notion of a {\it facet vector} of $F$~\cite{Min1897}. If $F$ is a $(d-2)$-dimensional face of $P$, then 
$F$ is standard iff $\belt(F)$ consists of 4 facets.

\subsection{Voronoi's Conjecture}\hspace*{\fill}

The following Conjecture~\ref{voronoi} has been posed in 1909, and it has not been proved or disproved so far in full generality.

\begin{conj}[{G.~Voronoi, \cite{Vor1908}}]\label{voronoi}
Every $d$-dimensional parallelohedron $P$ is a Dirichlet--Voronoi domain for $\Lambda(P)$ with respect to some Euclidean metric in $\mathbb E^d$.
\end{conj}

In the same paper~\cite{Vor1908} Voronoi proved his conjecture for primitive parallelohedra, i.e. for the case of every vertex of the tiling $T(P)$ being shared by exactly $d+1$ translates of $P$. Zhitomirskii~\cite{Zhi1929} proved it for $(d-2)$-primitive parallelohedra, i.e. for such parallelohedra that every $(d-2)$-dimensional face of $T(P)$ belongs to exactly $3$ tiles. The property of $P$ to be $(d-2)$-primitive is equivalent to the condition that
$P$ has no four-belts.

A.~Ordine \cite{Ord2005} proved the Voronoi's Conjecture for so called $3$-irreducible parallelohedra. Up to the moment, no improvements of this result are known.

In~\cite{Erd1999} R.~Erdahl has shown that Voronoi's Conjecture is true for space-filling zonotopes. The term {\it zonotope} denotes a Minkowski sum of 
several segments. This paper is an attempt to develop the theory of parallelohedra in this direction.

\begin{defn}
If a $d$-parallelohedron $P$ is a Dirichlet-Voronoi domain for some $d$-dimensio\-nal lattice, then $P$ is called a {\it Voronoi parallelohedron}.
\end{defn}

In the present paper we aim to prove Voronoi's Conjecture (Conjecture \ref{voronoi}) for parallelohedra obtained by taking a Minkowski sum $P+I$ of a
$d$-dimendional ($d\geq 4$) Voronoi parallelohedron and a segment $I$. \'A.~Horv\'ath in~\cite{Hor2007} calls $P+I$ an {\it extension} of $P$.

We need to recall several notions concerning lattice Delaunay tilings.

Let $\Lambda$ be a $d$-dimensional lattice in the Euclidean space $\mathbb E^d$. We call a sphere
$$S(\mathbf x, r) = \{\mathbf y \in \mathbb E^d : \|\mathbf y-x\| = r\}$$
{\it empty}, if $\|\mathbf z - \mathbf x\| \geq r$ for every $\mathbf z \in \Lambda$.

If $S(\mathbf x, r)$ is an empty sphere and
$$\dim \aff (S(\mathbf x, r) \cap \Lambda) = d,$$
then the set 
$$\conv (S(\mathbf x, r) \cap \Lambda)$$
is called a {\it lattice Delaunay $d$-cell}. It is known (see, for example,~\cite{Del1937}) that all lattice Delaunay $d$-cells
for a given lattice $\Lambda$ form a face-to-face tiling $D_{\Lambda}$ of $\mathbb E^d$.

Each $k$-face of a Delaunay $d$-cell is affinely equivalent to some Delaunay $k$-cell (see \cite[\S~13.2]{DLa1996} for details). 
Thus, for simplicity, we can call all faces of $D_{\Lambda}$ just Delanay cells.

There is a duality between the Delaunay tiling $D_{\Lambda}$ and the Voronoi tiling $V_{\Lambda}$. Namely, for every face $F$ of $V_{\Lambda}$ there
is a cell $D(F)$ of $D_{\Lambda}$ such that
\begin{enumerate}
	\item If $P$ is a $d$-parallelohedron, $D(P)$ is a center of $P$.
	\item $D(F)\subset D(F')$ iff $F'\subset F$.
\end{enumerate}

Let $F$ be a face of $T(P)$ and let $\dim F = d - k$. Consider a $k$-dimensional plane $p$ that intersects $F$ transversally. In a small neighborhood of $F$ the section of $T(P)$ by $p$ coincides with a complete $k$-dimensional polyhedral fan, which is called the {\it fan of a face $F$} and denoted by $\fan(F)$.
By duality, the combinatorics of $\fan(F)$ is uniquely determined by the combinatorics of $D(F)$ and vice versa.

We will particularly need the classification of Delaunay $k$-cells for $k = 2,3$ (or, equivalently, all possible structures of fans $\fan(F)$ of dimension
2 or 3). There are two possible combinatorial types of two-dimensional fans and five possible combinatorial types of three-dimensional fans. They are shown 
in Figure~\ref{duals}. These types are listed, for example, by B.~Delaunay \cite[\S 8]{Del1929}, who solved a more complicated problem --- to find all possible
combinatorial types of 3-dimensional fans $\fan(F)$ without assumption that $P$ is Voronoi.

\begin{figure}[!t]
\centerline{\includegraphics[height=3cm]{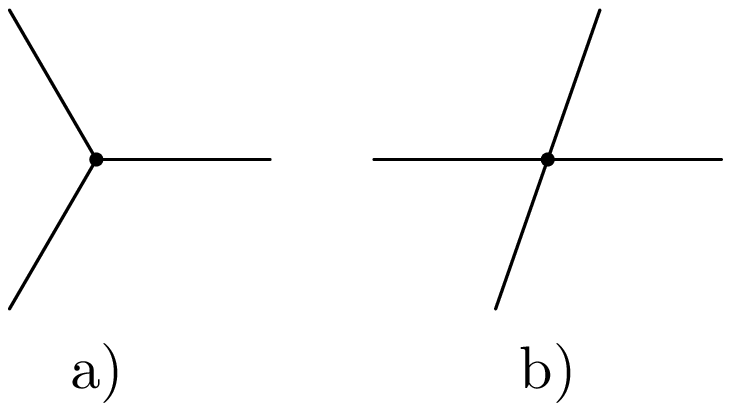}}

\vspace{5mm}
\centerline{\includegraphics[height=3cm]{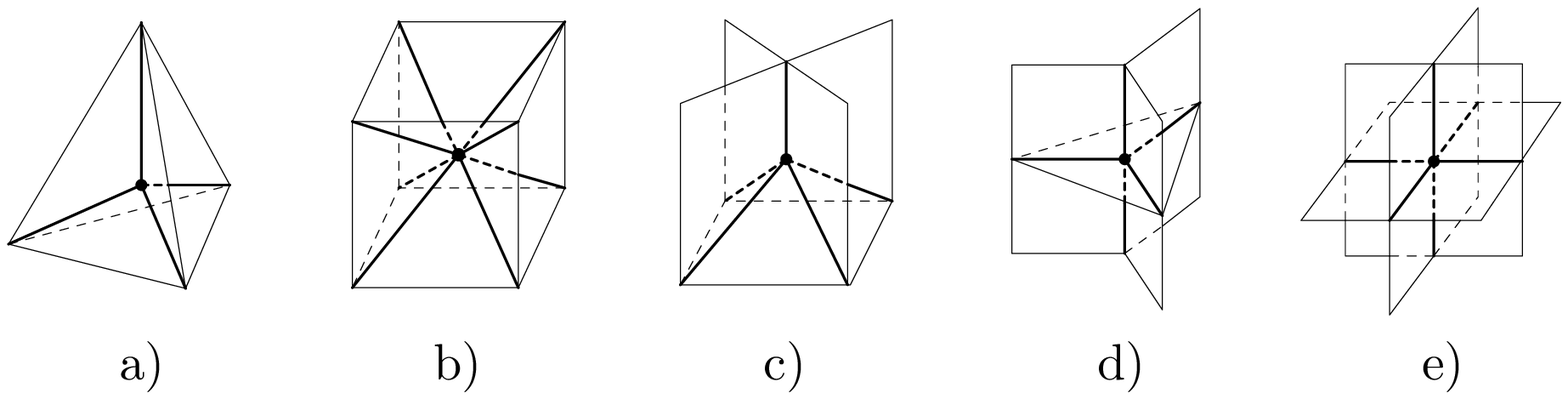}}

\vspace{3mm}
\caption{Fans of $(d-2)$- and $(d-3)$-faces} \label{duals}
\end{figure}

An explicit classification of all affine types of Delaunay $k$-cells exists for $k\leq 6$ [Dutour].

Notice that a $(d-2)$-face $F$ of a parallelohedron $P$ is standard iff it determines a four-belt, or (assuming that $P$ is a Voronoi parallelohedron) iff 
the dual Delaunay 2-cell $D(F)$ is a rectangle.

\subsection{Reducibility of parallelohedra}\hspace*{\fill}

\begin{defn} 
A parallelohedron $P$ is called {\it reducible}, if $P = P_1\oplus P_2$, where $P_1$ and $P_2$ are convex polytopes of smaller dimension. 
\end{defn}

From~\cite[Lemma~3 and Proposition~4]{Gav2013} it follows that $P_1$ and $P_2$ are parallelohedra and if $P$ is Voronoi, then $P_1$ and $P_2$ are Voronoi as well.

A.~Ordine proved the following criterion of reducibility for parallelohedra.

\begin{thm}[A.~Ordine, \cite{Ord2005}]\label{thm:ordine_reducibility}
Let $P$ be a parallelohedron. Suppose that each facets of $P$ is colored either with red or with blue so that
\begin{enumerate}
	\item[\rm 1.] Opposite facets of $P$ (with respect to the central symmetry of $P$) are of the same color.
	\item[\rm 2.] If two facets of $P$ belong to a common six-belt, then they are of the same color.
	\item[\rm 3.] Not all facets of $P$ are colored with the same color.  
\end{enumerate}
Then one can represent $P$ as $P_1\oplus P_2$ such that blue facets form $P_1\oplus\partial P_2$ and red facets form $\partial P_1\oplus P_2$.

The reversed statement also holds. Namely, if $P=P_1\oplus P_2$, assume that the facets of $P_1\oplus\partial P_2$ form the blue part of $\partial P$ and
$\partial P_1\oplus P_2$ form the red part. Then the resulting coloring satisfies conditions {\rm 1 -- 3}.
\end{thm}

We mention that the key to the proof of Theorem \ref{thm:ordine_reducibility} was the following.

\begin{thm}[A.~Ordine, \cite{Ord2005}]\label{thm:ordine_reducibility_misc}
Let $P$ be a parallelohedron in $\mathbb E^d$. Suppose that $q_1$ and $q_2$ are linear spaces of dimension at least 1 such that 
$q_1\oplus q_2 = \mathbb E^d$. Assume that for every facet $F\subset P$ holds
$$\mathbf s(F) \in q_1 \quad \text{or} \quad \mathbf s(F) \in q_2.$$
Then $P = P_1 \oplus P_2$ and $\lin \aff P_i = q_i$ for $i = 1,2$.
\end{thm}

This theorem plays an important role at the very end of this paper.

\subsection{Main results}\hspace*{\fill}

In this paper we prove the following three theorems simultaneously.

\begin{thm}\label{thm:voronoi_for_sum}
Let $I$ be a segment. Suppose that $P$ and $P+I$ are parallelohedra and $P$ is Voronoi in the standard Euclidean metric of $\mathbb E^d$. 
Then $P+I$ is Voronoi in some other Euclidean metric.
\end{thm}

\begin{thm}\label{thm:cross}
Let $P$ be a Voronoi parallelohedron in $\mathbb E^d$ and let $\Pi_1, \Pi_2$ be hyperplanes. Assume that for every facet $F\subset P$ holds
$$\mathbf s(F) \in \Pi_1 \quad \text{or} \quad \mathbf s(F) \in \Pi_2.$$
Then $P$ is reducible.
\end{thm}

\begin{thm}\label{thm:cross_reduction}
Let $P$ be a Voronoi parallelohedron in $\mathbb E^d$ and let $\Pi_1, \Pi_2$ be hyperplanes. Assume that the following conditions hold.
\begin{enumerate}
	\item[\rm 1.] $P = P_1 \oplus P_2 \oplus \ldots \oplus P_k$, where $k > 1$ and all $P_i$ are irreducible.
	\item[\rm 2.] $\mathbf s(F) \in \Pi_1$ or $\mathbf s(F) \in \Pi_2$ for every facet $F\subset P$.
\end{enumerate}
Then for each $i = 1, 2, \ldots, k$ one has $\aff P_i \parallel \Pi_1$ or $\aff P_i \parallel \Pi_2$
\end{thm}

Theorems \ref{thm:cross} and \ref{thm:cross_reduction} require that $P$ has a special property. Since this property is extremely important for us, we
give a definition.

\begin{defn}\label{def:cross}
Let $P$ be a parallelohedron in $\mathbb E^d$. We say that a pair of hyperplanes $(\Pi_1, \Pi_2)$ is a {\it cross} for $P$ if for every facet 
$F\subset P$ holds
$$\mathbf s(F) \in \Pi_1 \quad \text{or} \quad \mathbf s(F) \in \Pi_2.$$
\end{defn}

Theorems \ref{thm:cross} and \ref{thm:cross_reduction} are direct generalizations of Theorem \ref{thm:ordine_reducibility_misc} in the class of
Voronoi parallelohedra. In turn, Voronoi's Conjecture for space-filling zonotopes first proved by R.~Erdahl~\cite{Erd1999} is an immediately follows
by induction from Theorem~\ref{thm:voronoi_for_sum}.

\section{Free segments and free spaces of parallelohedra}

\begin{defn}
Let $P$ be a $d$-dimensional parallelohedron. Let $I$ be a segment such that $P+I$ is a $d$-dimensional parallelohedron as well. Then $I$ is
called a {\it free segment} for $P$.
\end{defn}

\begin{defn}
Let $P$ be a $d$-dimensional parallelohedron. A linear space $p$ is called a {\it free space} for $P$ if every segment $I\parallel p$ is free for $P$.
\end{defn}

We will extensively use the following criterion of free segments.

\begin{thm}[V.~Grishukhin, \cite{Gri2004}]\label{thm:1.2}
Let $P$ be a parallelohedron and $I$ be a segment. Then $I$ is free for $P$ if and only if every six-belt of $P$ contains a facet parallel to $I$.
\end{thm}

We mention that the proof of Theorem~\ref{thm:1.2} in~\cite{Gri2004} was incomplete. M.~Dutour noticed that not all belts of $P+I$ were checked
to have 4 or 6 facets. Namely, the belts spanned by $(d-2)$-faces of form $E\oplus I$, where $E$ is a $(d-3)$-face of $P$, were not considered.
The same remark refers to Theorem~\ref{thm:4.1} as well. However, the missing case is considered in~\cite{DGM2013}, where the complete proof of
Theorem~\ref{thm:1.2} is given, and the same case analysis gives the proof of Theorem~\ref{thm:4.1}. See also Lemma~\ref{lem:cases}.

Theorem \ref{thm:1.2} has an immediate corollary which motivates introducing the notion of free space.

\begin{cor}\label{cor:free_space}
Let $P$ be a parallelohedron and let $F_1, F_2, \ldots, F_k$ be facets of $P$ with the property that each six-belt of $P$ contains at least one $F_i$.
Then
\begin{equation}\label{eq:perfect_free}
\lin\aff F_1 \cap \lin\aff F_2 \cap \ldots \cap \lin\aff F_k
\end{equation}
is a free space for $P$.
\end{cor}

\begin{defn}
Let $P$ be a $d$-dimensional parallelohedron. A free space for $P$ of form (\ref{eq:perfect_free}) is called {\it perfect}.
\end{defn}

The notions and statements above concerning free segments and free spaces do not require that $P$ is Voronoi. If, however, $P$ is Voronoi, then

\begin{enumerate}
	\item $I$ is free for $P$ if and only if each triple of facet vectors corresponding to a six-belt contains a vector $\mathbf s(F)\, \bot\, I$.
	\item If $\mathbf s(F_1), \mathbf s(F_2), \ldots, \mathbf s(F_k)$ are facet vectors of $P$ and each triple of facet vectors corresponding to a
	  six-belt contains some $\mathbf s(F_i)$ or some $-\mathbf s(F_i)$, then the orthogonal complement
    \begin{equation}\label{eq:perfect_free_alt}
      \left\langle \mathbf s(F_1), \mathbf s(F_2), \ldots, \mathbf s(F_k) \right\rangle^{\bot}
    \end{equation}
    is a perfect free space for $P$.
\end{enumerate}

Now return from the Voronoi case to the case of general parallelohedra.

\begin{defn}
Let $P$ be a parallelohedron of dimension $d$ and let $I$ be a free segment for $P$. We call a $(d-2)$-dimensional face $F$ of $P$ {\it semi-shaded}
by $I$ if $F\oplus I$ is a facet of $P+I$.
\end{defn}

The following statement also immediately follows from \ref{thm:1.2}.

\begin{cor}
Let $P$ be a $d$-dimensional parallelohedron and let $I$ be a free segment for $P$. Then every $(d-2)$-dimensional face of $P$ semi-shaded by $I$ is standard.
\end{cor}

Consequently, the standard vector $\mathbf s(F)$ is defined for every $(d-2)$-dimensional face $F$ semi-shaded by $I$.

Introduce the notation
$$\mathcal A_I(P) = \{\mathbf s(F) : \dim\aff F = d-2 \; \text{and $F$ is semi-shaded by $I$} \},$$
$$\mathcal B_I(P) = \{\mathbf s(F) : \dim\aff F = d-1 \; \text{and $F\parallel I$} \}.$$

Working with the sets $\mathcal A_I(P)$ and $\mathcal B_I(P)$ we will need a theorem by B.~Venkov and a corollary emphasized by \'A.~Horv\'ath.
We provide both results below.

\begin{defn}
Let $P$ be a $d$-dimensional parallelohedron and $p$ be a linear space of dimension $d'$, where $0<d'<d$. Assume that for every point $\mathbf x \in \mathbb E^d$
the set $P \cap (\mathbf x+p)$ is either a $d'$-dimensional polytope or empty. Then we say that $P$ {\it has positive width} along $p$.
\end{defn}

\begin{thm}[B.~Venkov, \cite{Ven1959}]\label{thm:venkov_proj}
Assume that $P$ is a $d$-dimensional parallelohedron with positive width along a $d'$-dimensional linear space $p$. Let $F_1, F_2, \ldots, F_k$ be all 
facets of $P$ parallel to $p$ and let $\mathbf s_i = \mathbf s(F_i)$. Finally, let $\proj_p$ denote the projection along $p$ onto the complementary
space $q$. Then
\begin{enumerate}
	\item[\rm 1.] $\proj_p$ is a bijection of $\left\langle \mathbf s_1, \mathbf s_2, \ldots, \mathbf s_k \right\rangle$ and $q$. In particular,
	$$\dim \left\langle \mathbf s_1, \mathbf s_2, \ldots, \mathbf s_k \right\rangle = d-d'.$$ 
	\item[\rm 2.] The set 
	$$\{ \proj_p (P+\mathbf t) : \mathbf t \in \mathbb Z(\mathbf s_1, \mathbf s_2, \ldots, \mathbf s_k) \} $$
	is a face-to-face tiling of $q$ by parallelohedra. All tiles are translates of $\proj_p (P)$.
\end{enumerate}
\end{thm}

If $I$ is free for $P$, then $P+I$ has positive width along $I$. Therefore Theorem~\ref{thm:venkov_proj} has the following corollary.

\begin{cor}[\'A. Horv\'ath, \cite{Hor2007}]\label{cor:horvath}
Suppose $P$ is a $d$-dimensional parallelohedron and a segment $I$ is free for $P$. Then
$$\dim \left\langle \mathcal A_I(P) \cup \mathcal B_I(P) \right\rangle = d - 1.$$
In addition, if $\proj_I$ is a projection along $I$ onto a complementary $(d-1)$-space, then $Q = \proj_I(P)$ is a parallelohedron and
$$\Lambda(Q) = \proj_I \bigl(\mathbb Z(\mathcal A_I(P) \cup \mathcal B_I(P))\bigr).$$
\end{cor}

\section{Layering of parallelohedra with free segments}\label{sec:3}

\begin{defn}
Let $P$ be a $d$-dimensional parallelohedron and $I$ be a free segment for $P$.
Fix a vector $\mathbf e_{I} \parallel I$. Define {\it the cap of $P$ visible by $I$}, or, simply, {\it the $I$-cap of $P$} as a homogeneous $(d-1)$-dimensional complex $\pcap_{I}(P)$ consisting of all facets $F$ of $P$ satisfying the condition
$$\mathbf e_{I} \cdot \mathbf n(F) < 0$$
and all subfaces of those facets. (Obviously, each $I$ defines two caps centrally symmetric to each other.)
\end{defn}

For a parallelohedron $P$ and its free segment $I$ define
$$\mathcal C_I(P) = \{ \mathbf s(F): F \; \text{is a facet of}\; \pcap_I(P)\}.$$

\begin{lem}\label{lem:3.2}
Let $P$ be a parallelohedron and $I$ be its free segment. Then
$$\lin\aff \mathcal C_I(P) \subseteq \left\langle \mathcal A_I(P) \cup \mathcal B_I(P) \right\rangle.$$
\end{lem}

\begin{proof}
Let $I = [-\mathbf x, \mathbf x]$ and $\mathbf e_I = 2\mathbf x$.

The proof is by contradiction. Suppose that $F_1$ and $F_2$ are facets of $\pcap_I(P)$ and
$$\mathbf s(F_1) - \mathbf s(F_2) \notin \left\langle \mathcal A_I(P) \cup \mathcal B_I(P) \right\rangle.$$

One can easily see that for every $\lambda > 0$ the segment $\lambda I$ is free for $P$. Moreover, $F_1 + \lambda \mathbf x$ and
$F_2 + \lambda \mathbf x$ are facets of $P+\lambda I$ with facet vectors
$$ \mathbf s(F_1) + \lambda \mathbf e_I \quad \text{and} \quad \mathbf s(F_1) + \lambda \mathbf e_I $$
respectively. Hence 
$$\mathcal A_I(P) \cup \mathcal B_I(P) \cup \{ \mathbf s(F_1) - \mathbf s(F_2) \} \subset \Lambda (P+\lambda I).$$

By assumption, the lattice generated by $\mathcal A_I(P)$, $\mathcal B_I(P)$ and $\mathbf s(F_1) - \mathbf s(F_2)$ is $d$-dimensional and
does not depend on $\lambda$. Let $V$ be the fundamental volume of this lattice. Then the volume of $P+\lambda I$, which is the fundamental volume
of $\Lambda (P+\lambda I)$, is at most $V$. But as $\lambda \to \infty$, the volume of $P+\lambda I$ becomes arbitrtarily large, a contradiction.

\end{proof}

\begin{lem}\label{lem:3.3}
Let $P$ be a $d$-dimensional parallelohedron centered at $\mathbf 0$ and let $I$ be its free segment. Then the $(d-1)$-dimensional sublattice 
$$\mathbb Z(\mathcal A_I(P) \cup \mathcal B_I(P)) \subset \Lambda(P)$$
has index 1.
\end{lem}

\begin{proof}
Consider the sublattice $\Lambda_0 = \Lambda(P) \cap \left\langle \mathcal A_I(P) \cup \mathcal B_I(P) \right\rangle$.
It is enough to prove that
\begin{equation}\label{eq:3.1}
	\Lambda_0 = \mathbb Z(\mathcal A_I(P) \cup \mathcal B_I(P)).
\end{equation}

Assume that (\ref{eq:3.1}) does not hold.

Let $\mathbf t = \mathbf s(F)$, where $F$ is a facet of $\pcap_I(P)$. By Lemma \ref{lem:3.2},
$$\proj_{\lin \Lambda_0} (\mathbf s(F')) \in \{ \mathbf 0, \pm \mathbf t \}$$
for all $F'$ being facets of $P$. Here $\proj_{\lin \Lambda_0}$ is a projection along $\lin \Lambda_0$ onto $\mathbb R\cdot \mathbf t$.

Since the set of all facet vectors of $P$ generates $\Lambda(P)$,
$$\proj_{\lin \Lambda_0} \Lambda(P) = \mathbb Z \cdot \mathbf t.$$

Set $I = [-\mathbf x, \mathbf x]$, $\mathbf e_I = 2 \mathbf x$. Consider the tiling $T(P+\lambda I)$ for an arbitrary $\lambda > 0$. We will show that
$$\Lambda(P+\lambda I) = \Lambda_0 \oplus \mathbb Z \cdot (\mathbf t + \lambda \mathbf e_I).$$

To prove this, it is enough to check that all facet vectors of $P+\lambda I$ belong to 
$$\Lambda_0 \oplus \mathbb Z \cdot (\mathbf t + \lambda \mathbf e_I).$$
Indeed, each facet vector of $P+\lambda I$ is either from $\mathcal A_I(P)$, or from $\mathcal B_I(P)$, or it has the form
$$\pm (\mathbf s(F) + \lambda \mathbf e_I),$$
where $F$ is a facet of $P$ and the sign is chosen to be plus, if $F\in \pcap_I(P)$ and minus if $-F\in \pcap_I(P)$. In the first two cases
the facet vectors belong to $\Lambda_0$, and in the third case the facet vector belongs to $\pm(\Lambda_0 + \mathbf t + \mathbf e_I)$.

From Corollary \ref{cor:horvath} follows that for sufficiently large $\lambda$ the hyperplane $\aff \Lambda_0$ is covered, except for a lower-dimensional
subset, by interior parts of parallelohedra
$$\{ P + \lambda I + \mathbf u : \mathbf u \in \mathbb Z(\mathcal A_I(P) \cup \mathcal B_I(P)) \}.$$

Let $\mathbf v \in \Lambda_0 \setminus \mathbb Z(\mathcal A_I(P) \cup \mathcal B_I(P)$. Then the same holds for
$$\{ P + \lambda I + \mathbf v + \mathbf u : \mathbf u \in \mathbb Z(\mathcal A_I(P) \cup \mathcal B_I(P)) \}.$$
This is impossible since $T(P+\lambda I)$ is a tiling. Hence $\mathbf v$ does not exist and (\ref{eq:3.1}) holds.

\end{proof}

\begin{lem}\label{lem:3.4}
Let $P$ be a parallelohedron with a free segment $I$. Let $F$ be a facet of $P$ such that 
$$\mathbf s(F) \in \left\langle \mathcal A_I(P) \cup \mathcal B_I(P) \right\rangle.$$
Then $F$ is parallel to $I$.
\end{lem}

\begin{proof}
Assume the converse. Then $\mathbf s(F) \in \mathcal C_I(P)$. Thus 
$$\aff \mathcal C_I(P) \cap \aff (\mathcal A_I(P) \cup \mathcal B_I(P)) \neq \varnothing,$$
since the intersection contains $\mathbf s(F)$. Application of Lemma \ref{lem:3.2} gives
$$\left\langle \mathcal C_I(P) \right\rangle \subset \left\langle \mathcal A_I(P) \cup \mathcal B_I(P) \right\rangle.$$

This immediately gives
$$ \dim \left\langle \mathcal B_I(P) \cup \mathcal C_I(P) \right\rangle \leq d-1.$$

But $\mathcal B_I(P)$ together with $\mathcal C_I(P)$ generate a $d$-lattice $\Lambda(P)$, a contradiction.

\end{proof}

\begin{lem}\label{lem:3.5}
Let $P$ be a $d$-dimensional parallelohedron centered at $\mathbf 0$ and let $I$ be its free segment.
Choose a vector $\mathbf t$ so that 
$$\Lambda(P) = \mathbb Z(\mathcal A_I(P) \cup \mathcal B_I(P)) \oplus \mathbb Z \cdot \mathbf t.$$ 
Let $\mathbf v \in \mathbb Z(\mathcal A_I(P) \cup \mathcal B_I(P))$. Then
$$\proj_I (P\cap (P + \mathbf v + \mathbf t)) = \proj_I (P)\cap \proj_I (P + \mathbf v + \mathbf t).$$
\end{lem}

\begin{proof}
From Lemma \ref{lem:3.3} immediately follows that 
$$\mathcal C_I(P) \subset \pm \mathbf t + \mathbb Z(\mathcal A_I(P) \cup \mathcal B_I(P)).$$
Without loss of generality assume that the sign is ``+''.

Consider a homogeneous $(d-1)$-dimensional complex $\mathcal K$, all faces of which are faces of $T(P)$, and satisfying
$$|\mathcal K| = \bigcup\limits_{\mathbf v} (\pcap_I(P) + \mathbf v), $$
where $\mathbf v$ runs through the lattice $\mathbb Z(\mathcal A_I(P) \cup \mathcal B_I(P))$ and $|\mathcal K|$ denotes the support of $\mathcal K$.
Informally, $\mathcal K$ splits two layers
\begin{multline}\label{eq:3.2}
\mathcal L_0 = \{P+\mathbf v: \mathbf v\in \mathbb Z(\mathcal A_I(P) \cup \mathcal B_I(P))\} \quad \text{and} \\  
\mathcal L_1 = \{P+\mathbf t + \mathbf v: \mathbf v\in \mathbb Z(\mathcal A_I(P) \cup \mathcal B_I(P))\}.
\end{multline}

$\mathcal K$ has the following properties.
\begin{enumerate}
	\item The projection $\proj_I$ onto a hyperplane $\Pi$ transversal to $I$ is a homeomorphism between $|\mathcal K|$ and $\Pi$.
	\item $|\mathcal K| = \left(\bigcup\limits_{P'\in \mathcal L_0} P' \right) \bigcap \left(\bigcup\limits_{P'\in \mathcal L_1} P' \right)$.
	\item $\proj_I (P' \cap |\mathcal K|) = \proj_I(P')$ for every $P' \in \mathcal L_0 \cup \mathcal L_1$.
\end{enumerate}

Statement 1 follows from A.~D.~Alexandrov's tiling theorem~\cite{Ale1954}. We apply it to the complex spanned by polytopes
$$\{ F + \mathbf v : F \in \pcap_I(P), \; \mathbf v \in \mathbb Z(\mathcal A_I(P) \cup \mathcal B_I(P))\}.$$
One can easily check that the set of $(d-1)$-polytopes above locally forms a local tiling around each face of dimension $(d-3)$. Hence this
set is a tiling of an affine $(d-1)$-space.

Therefore each line parallel to $I$ is split by $|\mathcal K|$ into two rays, say, the lower and the upper one
with respect to some fixed orientation of $I$. We will call the union of all lower closed rays {\it the lower part of $\mathbb R^d$} and 
the union of all upper closed rays {\it the upper part of $\mathbb R^d$} with respect to $|\mathcal K|$.

To prove statement 2 notice that all parallelohedra of $\mathcal L_0$ lie in one (say, lower) part of $\mathbb R^d$, respectively, 
all parallelohedra of $\mathcal L_1$ lie in the upper part. Thus the intersection is contained in the intersection of lower and upper parts, i.e.
in $|\mathcal K|$. On the other hand, every point of $|\mathcal K|$ is an intersection of some parallelohedron from $\mathcal L_0$ and 
some parallelohedron from $\mathcal L_1$.

Statement 3 is an immediate corollary of definitions of a cap and $\mathcal K$.

In the notation of Lemma \ref{lem:3.5}, let $P' = P + \mathbf v + \mathbf t$. Thus 
$$P\in \mathcal L_0 \quad \text{and} \quad P'\in \mathcal L_1.$$

From statement 2 follows that
$$P \cap P' = (P\cap |\mathcal K|) \cap (P' \cap |\mathcal K|).$$

Since $\proj_I$ is a homeomorphism of $|\mathcal K|$, one has
\begin{multline*}
\proj_I (P \cap P') = \proj_I \bigl((P\cap |\mathcal K|) \cap (P' \cap |\mathcal K|)\bigr) = \\
\proj_I (P\cap |\mathcal K|) \cap \proj_I (P' \cap |\mathcal K|) = \proj_I (P) \cap \proj_I(P').
\end{multline*}

The last identity is due to statement 3.

\end{proof}

\section{Free segments and Voronoi's Conjecture}

The following two theorems stated by V.~Grishukhin characterize when the Minkowski sum $P+I$ of a Voronoi parallelohedron $P$ and a segment $I$ is a
Voronoi parallelohedron in some Euclidean metric, probably, distinct from the respective metric for $P$.

\begin{thm}[V.~Grishukhin, \cite{Gri2006}]\label{thm:4.1}
Let $P$ and $P+I$ be parallelohedra. Suppose that $P$ is Voronoi and irreducible. Then Voronoi's Conjecture holds for $P+I$
iff $I\,\bot\, \mathbf s(F)$ for every standard $(d-2)$-face $F$ such that $\mathbf s(F) \in \mathcal A_I(P)$. The orthogonality
$\bot$ is related to the Euclidean metric that makes $P$ Voronoi. 
\end{thm}

\begin{thm}[V.~Grishukhin, \cite{Gri2006}]\label{thm:4.2}
Let $P$ and $P+I$ be parallelohedra. Suppose that $P$ is Voronoi and reducible so that $P =  P_1\oplus P_2$. Define
$$I_1 = \proj_{\lin\aff P_2}(I), \qquad I_2 = \proj_{\lin\aff P_1}(I),$$
assuming that $\proj_{\lin\aff P_2}$ is a projection along $\lin\aff P_2$ onto $\aff P_1$ and similarly for $\proj_{\lin\aff P_1}$. Then
\begin{enumerate}
	\item[\rm 1.] $P_1+I_1$ and $P_2+I_2$ are parallelohedra.
	\item[\rm 2.] Voronoi's Conjecture holds for $P+I$ iff it holds for both $P_1+I_1$ and $P_2+I_2$.
\end{enumerate}
\end{thm}

Theorem~\ref{thm:4.1} has an important consequence. A very similar statement can be found without a proof in \cite[Theorem~3.18]{Veg2006}.

\begin{cor}\label{cor:4.3} Suppose that $P$ and $P+I$ are parallelohedra and $P$ is Voronoi. Then Voronoi's Conjecture for $P+I$ holds if
\begin{equation}\label{eq:4.1}
\dim \left\langle \mathcal B_I(P) \right\rangle = d - 1.
\end{equation} 

\end{cor}

\begin{proof}
From (\ref{eq:4.1}) and Corollary \ref{cor:horvath} follows that 
$$\mathcal A_I(P) \subset \left\langle \mathcal B_I(P) \right\rangle.$$

For the rest of the proof the orthogonality is related to the Euclidean metric that makes $P$ Voronoi.

$I$ is orthogonal to $\left\langle \mathcal B_I(P) \right\rangle$. Thus, in particular, $I$ is orthogonal to each vector of $\mathcal A_I(P)$.
Now Corollary \ref{cor:4.3} immediately follows from Theorem \ref{thm:4.1}.

\end{proof}

By combining Corollary~\ref{cor:4.3} and Theorem~\ref{thm:4.2}, we can give an equivalent restatement of Theorem~\ref{thm:voronoi_for_sum} as follows.
The proof of both Theorems~\ref{thm:voronoi_for_sum}~and~\ref{thm:reducibility_for_twodim_free} will be given later following the way explained in
Section~\ref{sec:6}.

\begin{thm}\label{thm:reducibility_for_twodim_free}
If a Voronoi parallelohedron $P$ has a 2-dimensional free space, then $P$ is reducible.
\end{thm}

\begin{lem}
Theorems~\ref{thm:voronoi_for_sum}~and~\ref{thm:reducibility_for_twodim_free} are equivalent.
\end{lem}

\begin{proof}
Indeed, let Theorem~\ref{thm:voronoi_for_sum} be true. Suppose that there exists an irreducible Voronoi parallelohedron $P$ with a 2-dimensional free space.
Then $P$ has a perfect free space $q$ of dimension at least 2. Further, there are finitely many possibilities for 
$\left\langle \mathcal A_I(P) \cup \mathcal B_I(P) \right\rangle$. Therefore there is only a finite number of directions for $I \parallel q$
such that
$$I\, \bot\, \left\langle \mathcal A_I(P) \cup \mathcal B_I(P) \right\rangle.$$
But from Theorems~\ref{thm:voronoi_for_sum} and \ref{thm:4.1} it follows that the orthogonality should hold for every $I \parallel q$. The contradiction
shows that Theorem~\ref{thm:reducibility_for_twodim_free} follows from Theorem~\ref{thm:voronoi_for_sum}.

Let Theorem~\ref{thm:reducibility_for_twodim_free} be true. If Theorem~\ref{thm:voronoi_for_sum} is false, consider a counterexample $P+I$ with the least 
possible dimension of $P$. If $P$ is reducible, then by Theorem~\ref{thm:4.2} either $P_1 + I_1$ or $P_2 + I_2$ is a smaller counterexample, which is a 
contradiction to the minimality. If $P$ is irreducible, then, by Theorem~\ref{thm:reducibility_for_twodim_free}, $P$ has no free spaces of dimension
greater than 1. Therefore $I$ is parallel to a perfect free line and the identity (\ref{eq:4.1}) holds. Hence $P+I$ is Voronoi, so it is not a counterexample
to Theorem~\ref{thm:voronoi_for_sum}. As a result, Theorem~\ref{thm:voronoi_for_sum} follows from Theorem~\ref{thm:reducibility_for_twodim_free}.

\end{proof}

\begin{lem}\label{lem:4.4}
Let $I$ be a segment and let $P$ and $P+I$ be Voronoi parallelohedra (possibly, for different Euclidean metrics). Then $\proj_I(P)$ is a 
Voronoi parallelohedron for every possible choice of the image space of the projection.
\end{lem}

\begin{proof}
Indeed, it is enough to prove Lemma~\ref{lem:4.4} for any image space of $\proj_I$, since changing the image space results in the affine transformation
of the projection.

Let $\Pi = \left\langle \mathcal A_I(P) \cup \mathcal B_I(P) \right\rangle$ be the image space of $\proj_I$.
There exists a Euclidean norm such that $P$ is Voronoi with respect to it and $I$ is orthogonal to $\Pi$. For the irreducible case it is a consequence
of Theorem~\ref{thm:4.1}, and for reducible parallelohedra see \cite[\S 9]{Gri2006}. 
Then $\proj_I(P)$ is Voronoi with respect to the restriction of $|\cdot |_P$ to $\Pi$ (see the details in~\cite[Proposition 5]{Gri2006}). 

\end{proof}

\section{Two-dimensional perfect free spaces}\label{sec:5}

In this section we study the following construction. Let $P$ be a Voronoi parallelohedron and let $p$ be a two-dimensional perfect free space of $P$.
This case is extremely important in our argument, so we aim to establish several consequences. 

We need some more notation. Define 
$$\mathcal B_p(P) = \{ \mathbf s(F) : F\; \text{is a facet of $P$ and}\; \mathbf s(F)\, \bot\, p \},$$ 
$$\mathcal A_p(P) = \{ \mathbf s(F) : F\; \text{is a standard $(d-2)$-face of $P$ and}\; \mathbf s(F)\, \bot\, p \}.$$
Here the orthogonality is related to the Euclidean metric that makes $P$ Voronoi. 

Since $\mathbf s(F)\, \bot\, F$, then from the definition of a perfect space follows that
$$\dim \left\langle \mathcal B_p(P) \right\rangle = d-2.$$
Therefore $\left\langle \mathcal B_p(P) \right\rangle$ is the orthogonal complement to $p$ and hence
$$\mathcal A_p(P) = \mathcal A_I(P) \cap \left\langle \mathcal B_p(P) \right\rangle $$
for every $I \parallel P$.

\begin{lem}\label{lem:5.1}
Let $P$ be a Voronoi parallelohedron and let $p$ be a two-dimensional perfect free space of $P$. Let $I$ be a segment rotating in $p$. Then
$I_0$ is parallel to a perfect line iff the hyperplane
$$\left\langle \mathcal A_I(P) \cup \mathcal B_I(P) \right\rangle $$
as a function of $I$ is discontinuous at $I = I_0$.
\end{lem}

\begin{proof}
Notice that for every $I \parallel p$ one has $\mathcal B_p(P) \subseteq \mathcal B_I(P)$.

Suppose $I_0$ is not parallel to a perfect line. Then $\dim \left\langle \mathcal B_{I_0}(P) \right\rangle < d-1$, so
$$\mathcal B_{I_0}(P) = \mathcal B_p(P).$$

The same holds for all $I$ close enough to $I_0$. In addition, for all $I$ close enough to $I_0$ holds
$$\mathcal A_{I_0}(P) = \mathcal A_I(P).$$

Thus the hyperplane $\left\langle \mathcal A_I(P) \cup \mathcal B_I(P) \right\rangle$ is the same for all $I$ close enough to $I_0$. 

Suppose that $I_0$ is parallel to a perfect line. Then $\dim \left\langle \mathcal B_{I_0}(P) \right\rangle = d-1$.

The hyperplane function $\left\langle \mathcal A_I(P) \cup \mathcal B_I(P) \right\rangle$ takes only finitely many values, as $P$ has
finitely many standard vectors. Therefore to prove the discontinuity of this function it is enough to prove
\begin{equation}\label{eq:5.1}
\left\langle \mathcal A_I(P) \cup \mathcal B_I(P) \right\rangle \neq \left\langle \mathcal A_{I_0}(P) \cup \mathcal B_{I_0}(P) \right\rangle,
\end{equation}
if $I$ is close enough to $I_0$, but $I \nparallel I_0$.

Indeed, $P$ has a facet $F$ such that 
\begin{equation}\label{eq:5.2}
\mathbf s(F) \in \mathcal B_{I_0}(P) \setminus \mathcal B_p(P).
\end{equation}

If $I$ satisfies the conditions above, then $I \nparallel F$. But then, by Lemma~\ref{lem:3.4},
\begin{equation}\label{eq:5.3}
\mathbf s(F) \notin  \left\langle \mathcal A_I(P) \cup \mathcal B_I(P) \right\rangle.
\end{equation}

To finish the proof we observe that (\ref{eq:5.1}) follows from comparing (\ref{eq:5.2}) and (\ref{eq:5.3}).

\end{proof}

\begin{lem}\label{lem:5.2}
Let $P$ be a Voronoi parallelohedron and let $p$ be a two-dimensional perfect free space of $P$. Then
\begin{enumerate}
	\item[\rm 1.] $p$ contains exactly two perfect lines --- $\ell_1$ and $\ell_2$.
	\item[\rm 2.] Every facet $F$ of $P$ is parallel either to $\ell_1$ or to $\ell_2$, or to both. (The last case means $F\parallel p$.)
\end{enumerate}
\end{lem}

\begin{proof}

Start with the proof of statement 1. Choose a segment $I_0 \parallel p$ such that $\dim \left\langle \mathcal B_{I_0}(P) \right\rangle = d-2$. It is possible,
moreover, $I_0$ can be chosen arbitrarily, except for a finite number of directions.

Let $G$ be a standard $(d-2)$-face of $P$ such that
$$\mathbf s(G) \in \mathcal A_{I_0}(P) \setminus \mathcal A_p(P).$$

$G$ adjoins two facets $F$ and $F'$. These facets belong to antipodal caps of $P$ with respect to $I_0$, so
$$\mathbf s(F), \mathbf s(F') \notin \left\langle \mathcal A_{I_0}(P) \cup \mathcal B_{I_0}(P) \right\rangle.$$

Let $G'$ be the $(d-2)$-face of $P$ defined by
$$\belt(G') = \belt(G) \quad \text{and} \quad \inter G' \subset \inter \pcap_{I_0}(P).$$

Rotating the segment $I \parallel p$, one can observe that
$$\left\langle \mathcal A_I(P) \cup \mathcal B_I(P) \right\rangle = \left\langle \mathcal B_p(P) \cup \{\mathbf v \}\right\rangle,$$
where $\mathbf v \in \{ \mathbf s(F), \mathbf s(F'), \mathbf s(G), \mathbf s(G') \}$ and the cases $\{ \mathbf s(F), \mathbf s(F')\}$ happen only if
$I \parallel F$ and $I \parallel F'$ respectively. Thus $I \parallel F$ and $I \parallel F'$ are the only cases of discontinuity of
$$\left\langle \mathcal A_I(P) \cup \mathcal B_I(P) \right\rangle.$$

Therefore $p$ contains exatly two perfect lines, namely those parallel to $F$ and $F'$.

To prove statement 2 suppose that $F$ is a facet of $P$ and $F\nparallel p$. Let $I$ be a segment satisfying $I \parallel F$ and
$I \parallel p$. We have
$$\dim \left\langle \mathcal B_I(P) \right\rangle \geq \dim \left\langle \mathcal B_p(P) \cup \mathbf s(F) \right\rangle = d-1.$$

Hence $I$ is parallel to a perfect line. By assumption, $F \parallel I$, so $F$ is parallel to a perfect line.

\end{proof}

In the following lemma we will reproduce from \cite{DGM2013} the analysis of all possible arrangements of a free segment and a $(d-3)$-face of a parallelohedron.
Additionally, we will emphasize the arrangements that appear if a segment is parallel to a perfect two-dimensional free plane transversal to a $(d-3)$-face
or to a perfect line of such plane.

\begin{lem}\label{lem:cases}
Let $E$ be a $(d-3)$-face of a parallelohedron $P$ and let $I$ be a free segment for $P$. Let the image space of $\proj_E$ be the 3-space
where $\fan (E)$ lies. Then 
\begin{enumerate}
	\item[\rm 1.] If $\proj_E (I)$ does not degenerate into a point, then $\proj_E (I)$ and $\fan (E)$ are arranged together in one of the ways shown in 
	Figures~\ref{segm1}~--~\ref{segm2}. 
	\item[\rm 2.] If $p$ is a perfect two-dimensional free plane of $P$ and $p$ is transversal to $E$, then $\proj_E(p)$ is arranged with $\fan (E)$
	in one of the ways shown in Figure~\ref{planes}.
	\item[\rm 3.] Finally, if $I$ is parallel to a perfect free line in $p$, then $\proj_E (I)$ is arranged as one of the highlighted segments in 
	Figure~\ref{planes}.
\end{enumerate}
\end{lem}

\begin{proof}
Item 1 is verified by inspection. One should check if the condition of Theorem~\ref{thm:1.2} holds for all six-belts associated with $E$. For the proof
of item 2, one should enumerate all the 2-planes in the image space of $\proj_E$ such that all segments parallel to such a plane are enlisted in 
Figures~\ref{segm1}~--~\ref{segm2}. Finally, in order to select segments parallel to perfect free lines, one should apply Lemma~\ref{lem:5.2}. 

\end{proof}

\begin{figure}[!t]
\centerline{%
\begin{minipage}{0.49\textwidth}
\centerline{\includegraphics[height=4.5cm]{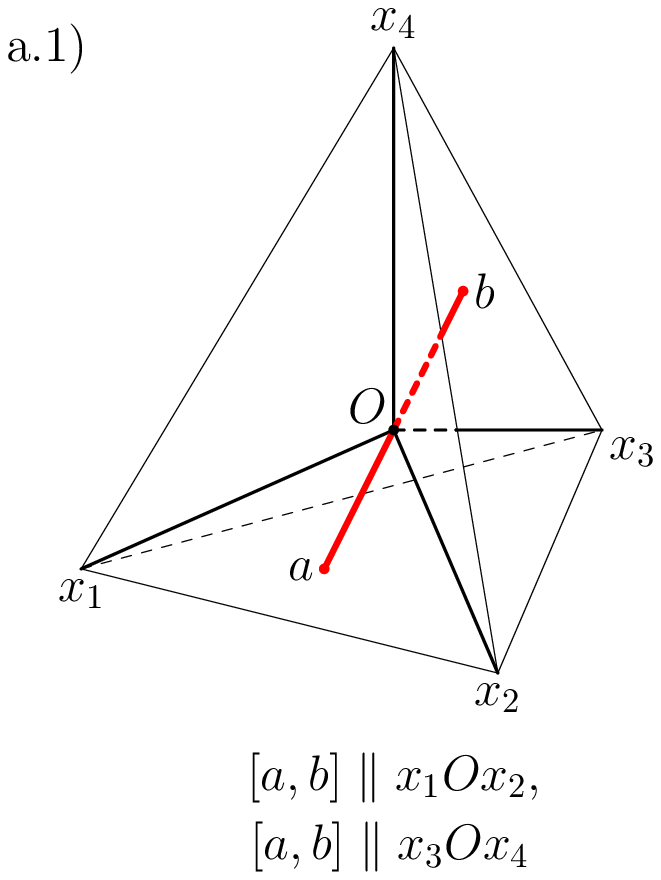}}
\end{minipage}%
\begin{minipage}{0.49\textwidth}
\centerline{\includegraphics[height=4.5cm]{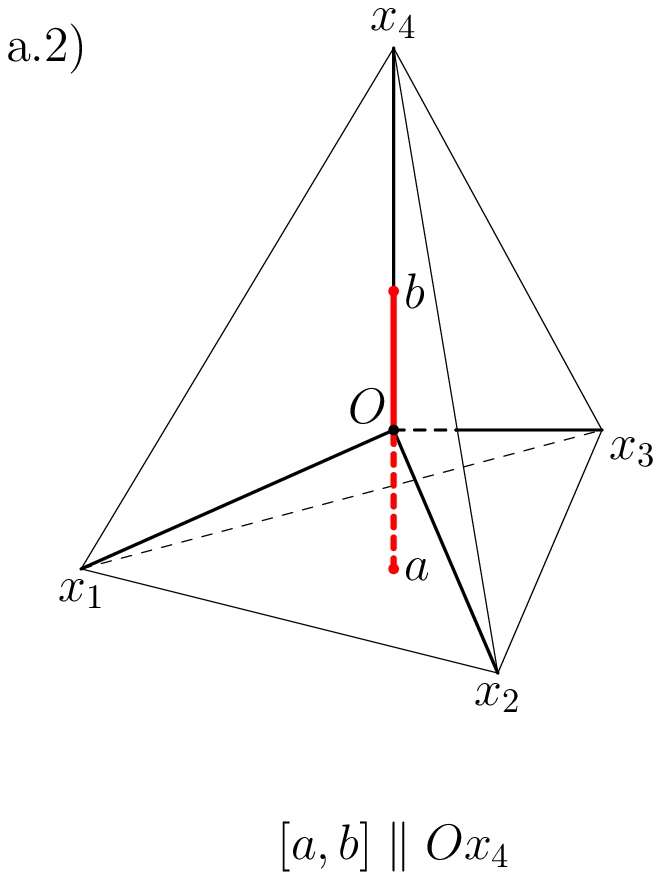}}
\end{minipage}%
}

\vspace{5mm}
\centerline{%
\begin{minipage}{0.49\textwidth}
\centerline{\includegraphics[height=4.5cm]{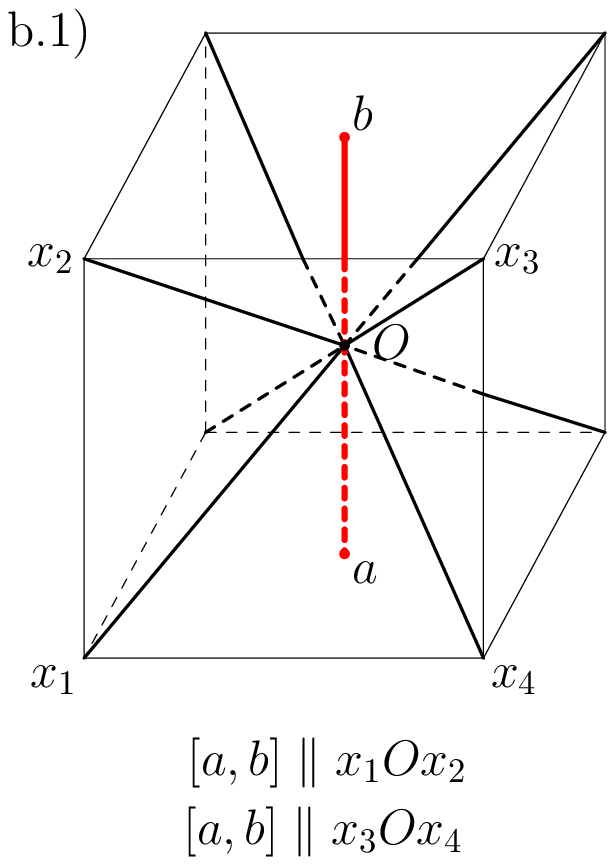}}
\end{minipage}%
\begin{minipage}{0.49\textwidth}
\centerline{\includegraphics[height=4.5cm]{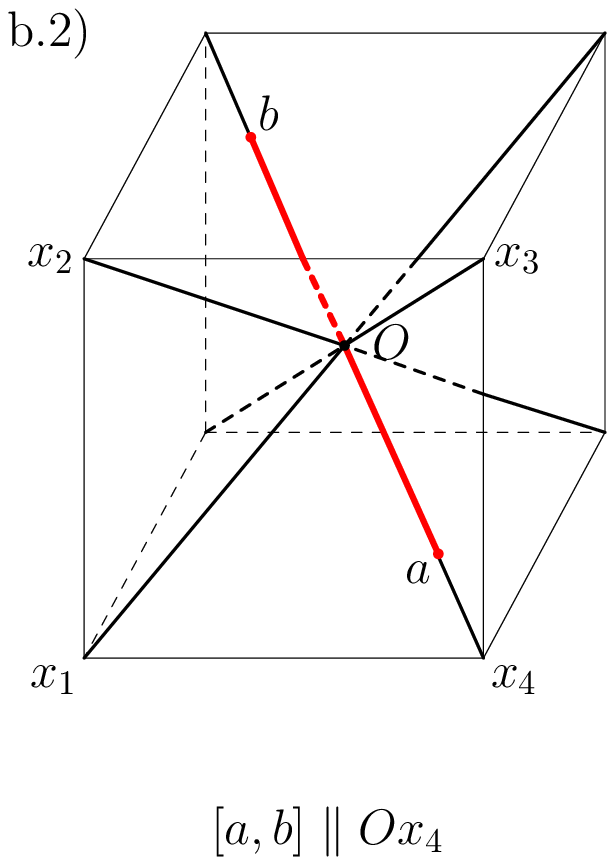}}
\end{minipage}%
}

\vspace{5mm}
\centerline{%
\begin{minipage}{0.32\textwidth}
\centerline{\includegraphics[height=4.5cm]{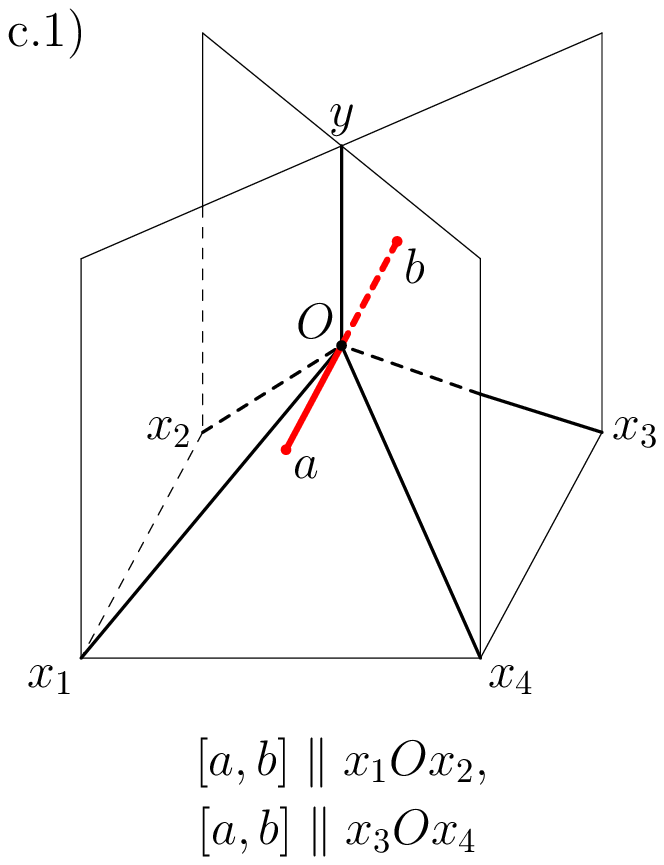}}
\end{minipage} \qquad%
\begin{minipage}{0.32\textwidth}
\centerline{\includegraphics[height=4.5cm]{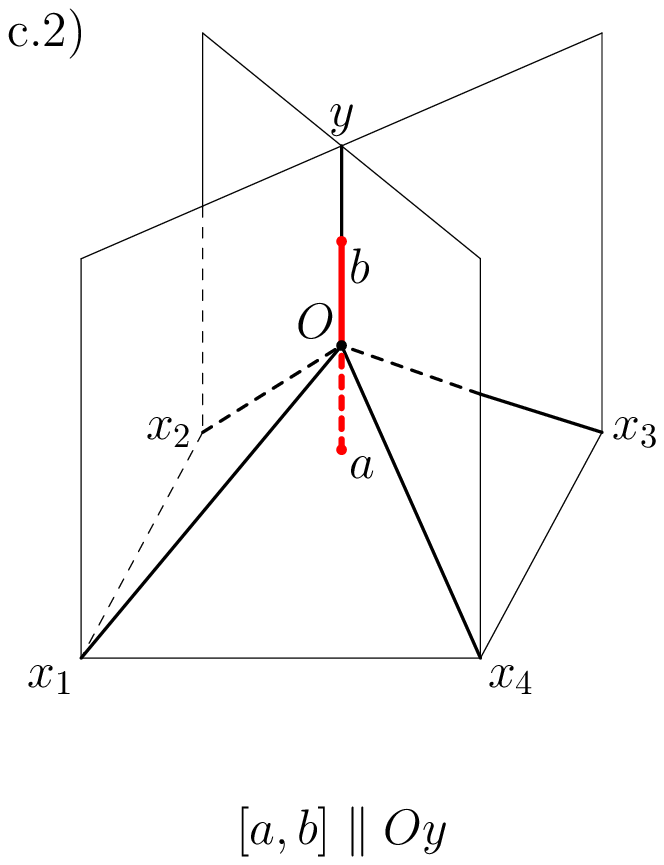}}
\end{minipage}%
\begin{minipage}{0.32\textwidth}
\centerline{\includegraphics[height=4.5cm]{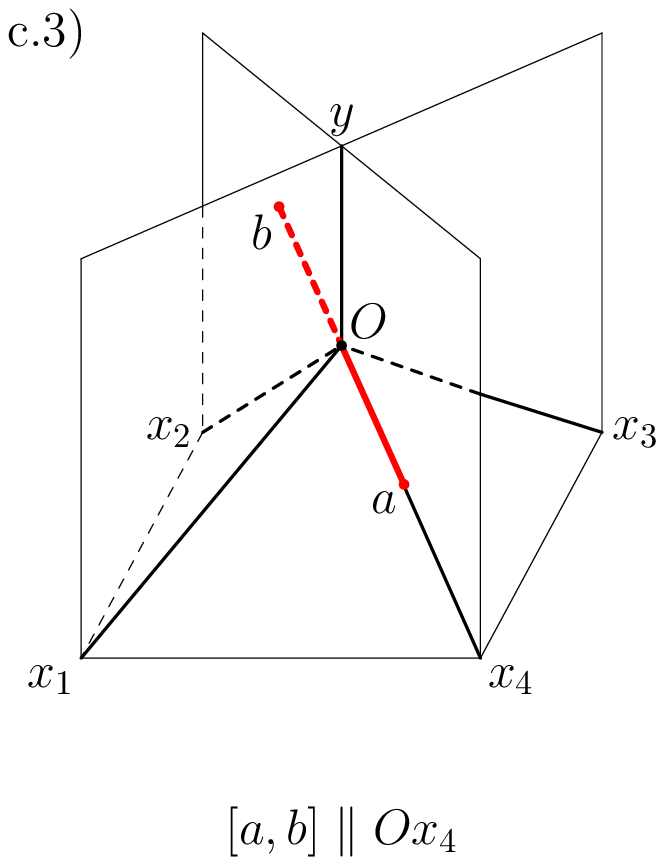}}
\end{minipage}%
}

\vspace{3mm}
\caption{Possible arrangements of free segments and $(d-3)$-faces}
\label{segm1}
\end{figure}

\begin{figure}[!t]
\centerline{%
\begin{minipage}{0.32\textwidth}
\centerline{\includegraphics[height=4.5cm]{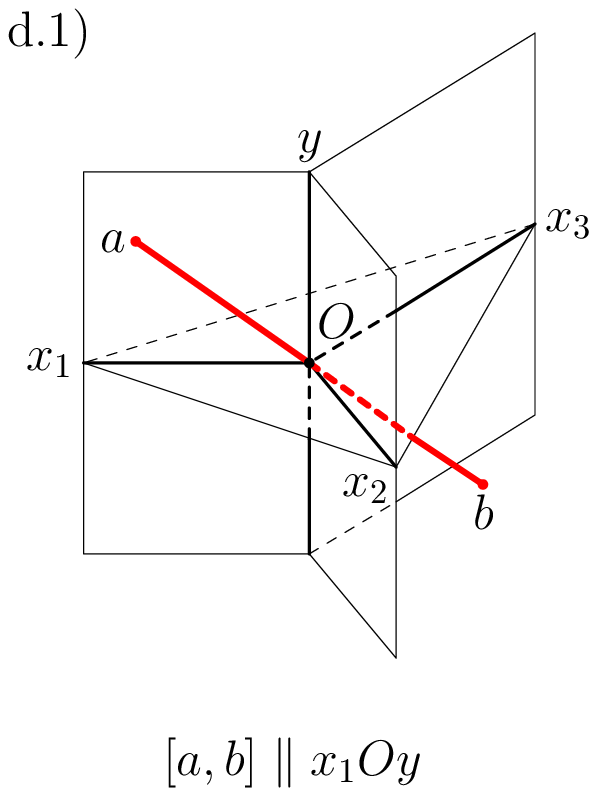}}
\end{minipage} \qquad%
\begin{minipage}{0.32\textwidth}
\centerline{\includegraphics[height=4.5cm]{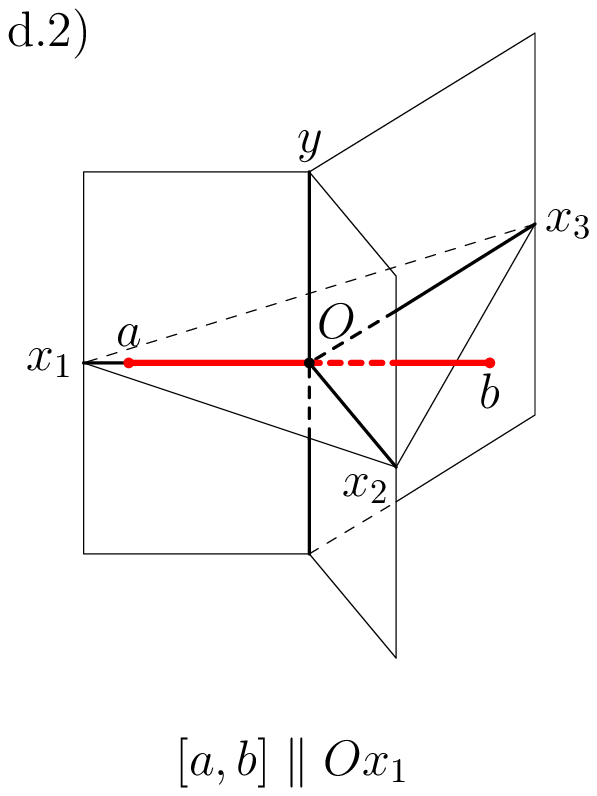}}
\end{minipage}%
\begin{minipage}{0.32\textwidth}
\centerline{\includegraphics[height=4.5cm]{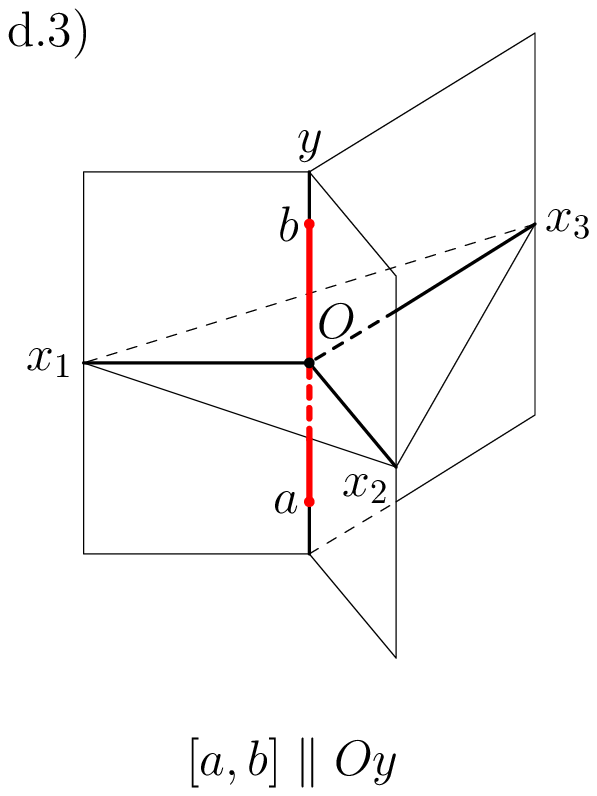}}
\end{minipage}%
}

\vspace{5mm}
\centerline{%
\begin{minipage}{0.32\textwidth}
\centerline{\includegraphics[height=4.5cm]{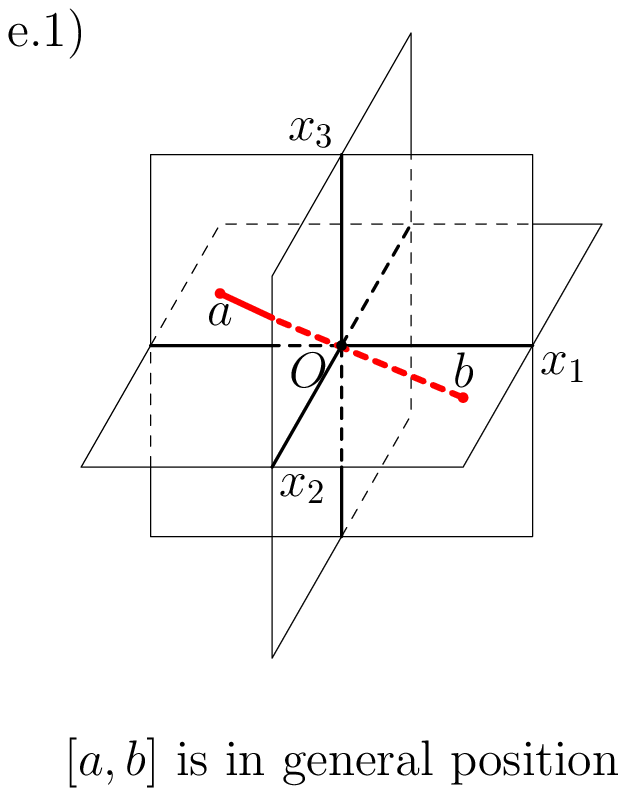}}
\end{minipage} \qquad%
\begin{minipage}{0.32\textwidth}
\centerline{\includegraphics[height=4.5cm]{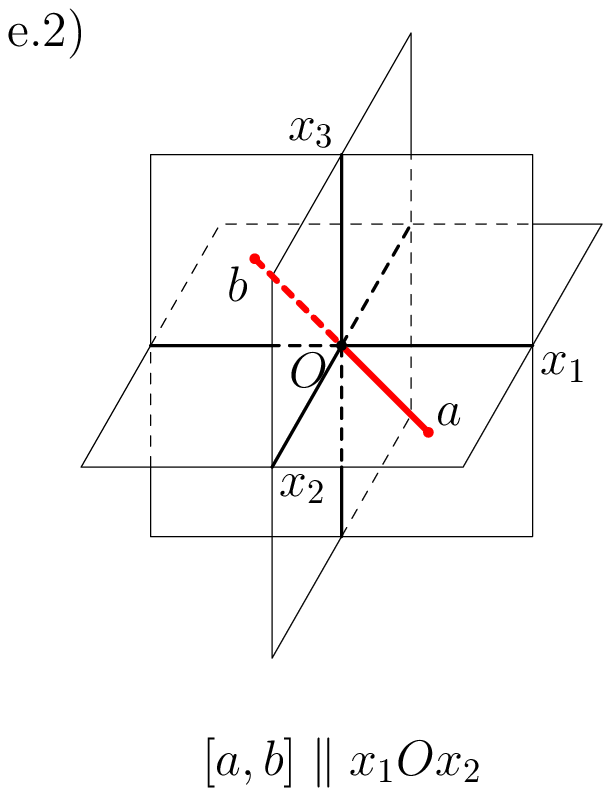}}
\end{minipage}%
\begin{minipage}{0.32\textwidth}
\centerline{\includegraphics[height=4.5cm]{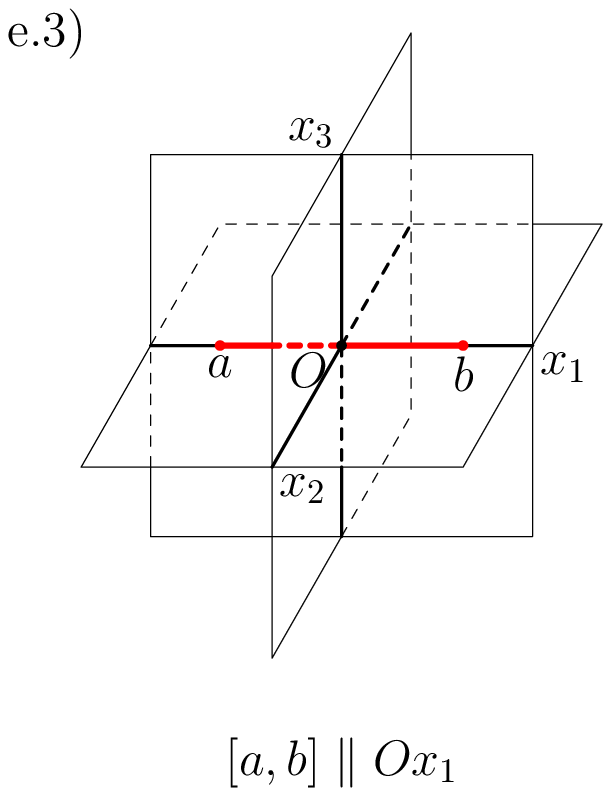}}
\end{minipage}%
\vspace{3mm}
}
\caption{Possible arrangements of free segments and $(d-3)$-faces, continued}
\label{segm2}
\end{figure}

\begin{figure}[!t]
\centerline{%
\begin{minipage}{0.32\textwidth}
\centerline{\includegraphics[height=4.5cm]{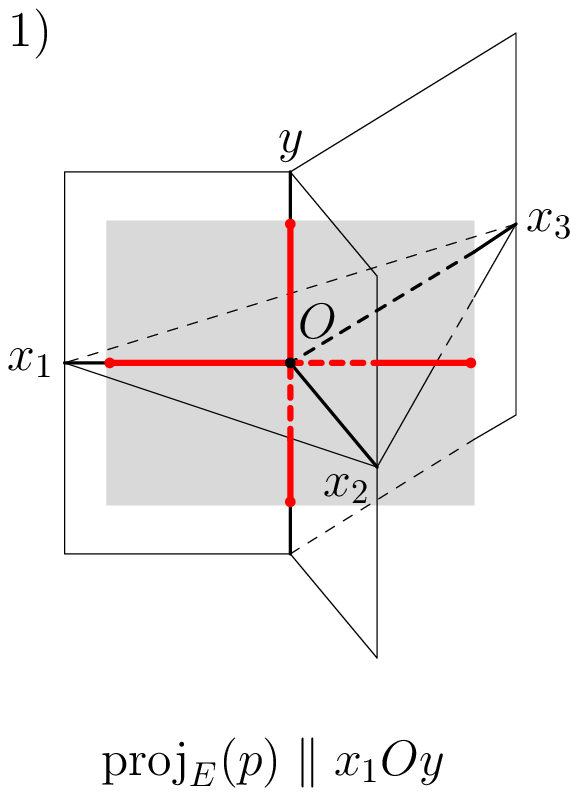}}
\end{minipage} \qquad%
\begin{minipage}{0.32\textwidth}
\centerline{\includegraphics[height=4.5cm]{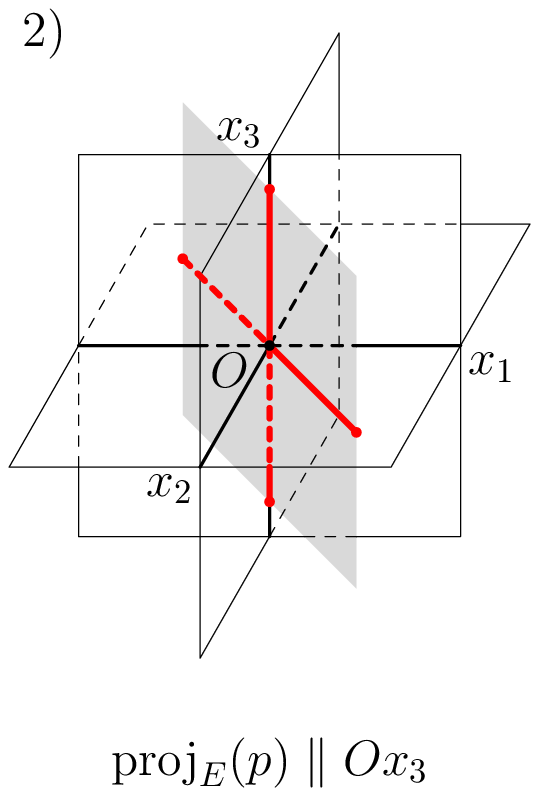}}
\end{minipage}%
\begin{minipage}{0.32\textwidth}
\centerline{\includegraphics[height=4.5cm]{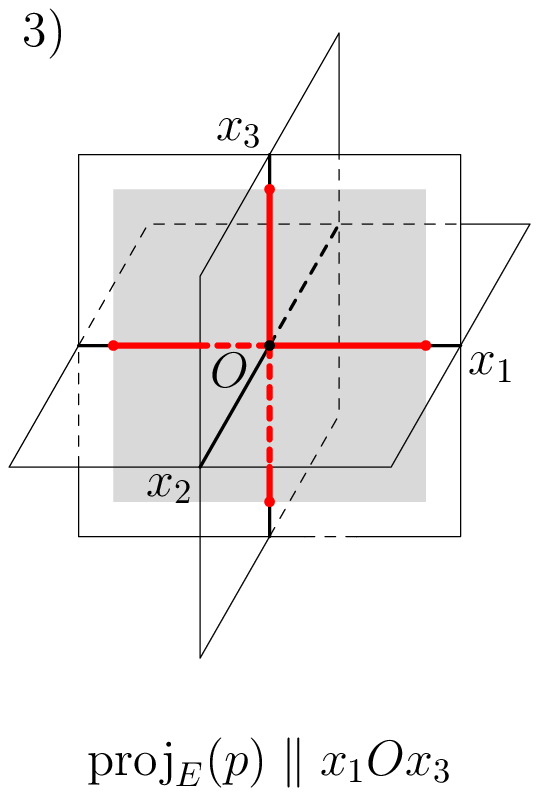}}
\end{minipage}%
}

\vspace{3mm}
\caption{Possible arrangements of $(d-3)$-faces and transversal free planes}
\label{planes}
\end{figure}

\begin{lem}\label{lem:5.3}
Let $P$ be a Voronoi parallelohedron and let $p$ be a two-dimensional perfect free space of $P$. In the notation of Lemma~\ref{lem:5.2} let
$I$, $Y_1$, $Y_2$ be segments such that $I\parallel p$, $Y_1 \parallel \ell_1$, $Y_2\parallel \ell_2$. Then
\begin{enumerate}
	\item[\rm 1.] If $G$ is a standard $(d-2)$-face of $P$ and $\mathbf s(G) \in \mathcal A_{Y_1}(P)$, then 
	$\mathbf s(G) \in \mathcal A_p(P)$.
	\item[\rm 2.] $P+Y_1$ is Voronoi in some Euclidean metric.	
	\item[\rm 3.] $p$ is a perfect free plane and $\ell_1$ and $\ell_2$ are perfect free lines of $P+Y_1$.
	\item[\rm 4.] $P+Y_1+Y_2+I$ is a parallelohedron.
\end{enumerate}
\end{lem}

\begin{proof}
If, on the contrary, $\mathbf s(G) \notin \mathcal A_p(P)$, then, by the argument of Lemma~\ref{lem:5.2}, $I = Y_1$ is the continuity point of 
$$\left\langle \mathcal A_I(P) \cup \mathcal B_I(P) \right\rangle $$
as a function of $I$. But $Y_1$ is parallel to a perfect line, so by Lemma~\ref{lem:5.1}, $I=Y_1$ is not a continuity point. A contradiction
gives statement 1.

Statement 2 is an immediate consequence of the definition of a perfect line and Corollary~\ref{cor:4.3}.

For the proof of statement 3, we first need to prove that every six-belt of $P+Y_1$ contains a facet parallel to $p$. Consider several cases.

Let a six-belt of $P+Y_1$ be inherited from a six-belt of $P$. Since $p$ is a perfect space and $Y_1 \parallel p$, then indeed such a six-belt 
contains a facet parallel to $p$. 

Let a six-belt of $P+Y_1$ be inherited from a four-belt of $P$. Such a six-belt contains a facet $G\oplus Y_1$, where $G$ is a standard $(d-2)$-face of
$P$. $G$ spans a four-belt of $P$ with no facet of this belt parallel to $Y_1$. According to Lemma~\ref{lem:5.2}, statement 2, all facets of the four-belt
of $P$ spanned by $G$ are parallel to $Y_2$. Thus $G \parallel Y_2$. As a result we have 
$$G\oplus Y_1 \parallel Y_1 \quad \text{and} \quad G\oplus Y_1 \parallel Y_2,$$
and hence $G\oplus Y_1 \parallel p$.

The last possible case occurs if a six-belt of $P+Y_1$ is spanned by a $(d-2)$-face of form $E\oplus Y_1$, where $E$ is a $(d-3)$-face of $P$. In this case $E$
and $Y_1$ can be arranged in the following ways reflected in Figures~\ref{segm1}~--~\ref{segm2}: c.1), d.1), e.1) (here we refer to Lemma~\ref{lem:cases}). 

Consider two subcases. First let $p$ not be transversal to $E$. Then $\ell = p \cap \lin \aff E$ is a line. Consider an arbitrary facet of the six-belt of 
$P+Y_1$ spanned by $E\oplus Y_1$. It has one of the forms $F + Y_1$ or $G\oplus Y_1$, where $F$ is a facet of $P$, respectively, $G$ is a standard
$(d-2)$-face of $P$. This facet is parallel to $E$ and therefore to $\ell$. Besides, it is parallel to $Y_1$. So it is parallel to $p$.
Consequently, $\mathbf s(F)$ or $\mathbf s(G)$ is orthogonal to $p$.

In the second subcase $p$ is transversal to $E$. $P+Y_1$ has a $(d-2)$-face $E\oplus Y_1$ only if the arrangement of $p$ corresponds to the case 2) in
Figure~\ref{planes} and the arrangement of $Y_1$ corresponds to the case e.2) in Figure~\ref{segm2}. But then $E\oplus Y_1$ spans a four-belt of $P$. Hence
no six-belt is possible in this subcase.

In addition notice that each of the lines $\ell_1$ and $\ell_2$ is parallel to more facets of $P+Y_1$ than a generic line in $p$. Hence
$\ell_1$ and $\ell_2$ are perfect free lines for $P+Y_1$.

By the same argument applied to $P+Y_1$, the parallelohedron $P+Y_1+Y_2$ is Voronoi and has $p$ as a perfect free space. Thus
$P+Y_1+Y_2+I$ is a parallelohedron, which is exactly statement 4.

\end{proof}

\begin{lem}\label{lem:5.4}
Let $P$ be a Voronoi parallelohedron and let $p$ be its perfect two-dimensional free space. Then $\proj_p (P)$ is a $(d-2)$-dimensional Voronoi
parallelohedron. 
\end{lem}

\begin{proof}
In fact, we want to check that $P$, $P+Y_1$, $\proj_{Y_1}(P)$ and $\proj_{Y_1}(P+Y_2)$ are Voronoi parallelohedra.

$P$ is Voronoi by assumption. $P+Y_1$ is Voronoi by Lemma~\ref{lem:5.3}, statement~2. Application of Lemma~\ref{lem:4.4} gives that
$\proj_{Y_1}(P)$ is Voronoi. Further, by the argument of Lemma~\ref{lem:5.3}, $P+Y_2$ is Voronoi and $\ell_1$ is its perfect free line.
Thus statement 2 of Lemma~\ref{lem:5.3} gives that $P+Y_1+Y_2$ is Voronoi. Lemma~\ref{lem:4.4} applied to $P+Y_2$ and $P+Y_1+Y_2$ implies that 
$\proj_{Y_1}(P+Y_2)$ is Voronoi. It remains to apply Lemma~\ref{lem:4.4} for the third time --- to $\proj_{Y_1}(P)$ and
$$\proj_{Y_1}(P)+\proj_{Y_1}(Y_2) = \proj_{Y_1}(P+Y_2).$$

\end{proof}

Let $I$ be a segment parallel to $p$, but not parallel to $Y_1$ and $Y_2$. For $j = 1,2$ let
$$\mathcal C^j_I(P) = \{ \mathbf s(F) : F\parallel Y_j\; \text{and}\; \mathbf s(F)\in \mathcal C_I(P) \}.$$
In other words, 
$$\mathcal C^j_I(P) = \mathcal C_I(P) \cap \mathcal B_{Y_j} (P).$$

The last formula immediately implies the following.

\begin{lem}\label{lem:5.5}
$\lin\aff \mathcal C^j_I(P) \parallel \left\langle \mathcal B_p (P) \right\rangle$ for $j = 1,2$.
\end{lem}

\begin{proof}
Indeed, 
$$\lin\aff \mathcal C^j_I(P) \subseteq  \left\langle \mathcal A_I(P) \cup \mathcal B_I(P) \right\rangle \cap 
\left\langle \mathcal A_{Y_j}(P) \cup \mathcal B_{Y_j}(P) \right\rangle.$$
The right part is an intersection of two different hyperplanes, each parallel to $\left\langle \mathcal B_p (P) \right\rangle$.
Thus the intersection is exactly $\left\langle \mathcal B_p (P) \right\rangle$.

\end{proof}

\begin{lem}\label{lem:5.6}
Let $P$ be a Voronoi parallelohedron and let $p$ be its perfect two-dimensional free space. In addition, let $P$ be centered at $\mathbf 0$. 
Assume that $I$ is a segment parallel to $P$, but not parallel to $Y_1$ and $Y_2$. With $\mathcal C^j_I(P)$ defined as above, choose
$$\mathbf w_j \in \mathcal C^j_I(P) \quad \text{for} \quad j=1,2 \quad \text{and}$$
$$\mathbf t_j \in \Lambda(P) \cap \left(\mathbf w_j + \left\langle \mathcal A_p(P) \cup \mathcal B_p(P) \right\rangle \right).$$
Then, if $P_j = P+\mathbf t_j$, one has
$$\proj_p (P\cap P_1 \cap P_2) = \proj_p(P) \cap \proj_p(P_1) \cap \proj_p(P_2).$$
\end{lem}

\begin{proof}
Consider the complex $\mathcal K$ defined in Section~\ref{sec:3}. Recall that $\mathcal K$ splits two layers 
$\mathcal L_0, \mathcal L_1 \subset T(P)$ given by formulae~(\ref{eq:3.2}). Since $P\in \mathcal L_0$ and $P_1, P_2\in \mathcal L_1$,
$$P\cap P_1 \cap P_2 \subset |\mathcal K|.$$

Set $Q = \proj_I(P)$. $Q$ is a Voronoi parallelohedron with a free segment $\proj_I(Y_1)$. (Or $\proj_I(Y_2)$, which has the same direction.)
One can easily see that the sets
$$\mathcal M_j = \{Q+\proj_I \mathbf t: \mathbf t \in \aff \mathcal C^j_I(P) \cap \Lambda(P)\} \quad \text{for} \quad j=1,2 $$
compose two layers of the same tiling of $\mathbb R^{d-1}$ by translates of $Q$. The notion of layers is the same as described in Section~\ref{sec:3}.
Call them {\it $\mathcal M_1$- and $\mathcal M_2$-layers}, respectively.

These layers are neighboring. Indeed, choose an arbitrary standard $(d-2)$-face $G$ of $P$ with
$$\mathbf s(G) \in \mathcal A_I(P) \setminus \mathcal A_p(P).$$
Then take a face $G'$ spanning the same belt as $G$ such that $\inter G' \subset \inter \pcap_I(P)$. It is not hard to see that
$\proj_I(G')$ belongs to the common boundary of the $\mathcal M_1$- and $\mathcal M_2$-layers. Consequently,
\begin{multline*}
\proj_p (P_1\cap P_2 \cap \mathcal K) = \proj_p ((P_1\cap \mathcal K)\cap (P_2 \cap \mathcal K)) = \\
\proj_{\proj_I(Y_1)} ((Q+\proj_I(\mathbf t_1)) \cap  (Q+\proj_I(\mathbf t_2))) = \\
\proj_{\proj_I(Y_1)} (Q+\proj_I(\mathbf t_1)) \cap \proj_{\proj_I(Y_1)} (Q+\proj_I(\mathbf t_2)) = \\
\proj_p (P_1) \cap \proj_p (P_2).
\end{multline*}

It remains to prove that
$$P_1\cap P_2 \cap \mathcal K \cap \proj_I^{-1}(P) \subset P.$$
So, it suffices to show that the the common boundary of the $\mathcal M_1$- and $\mathcal M_2$-layers separates the $\proj_I(Y_1)$-caps
of $Q$ from each other. It follows from the fact that each of these two caps is covered by its layer --- one by the $\mathcal M_1$-layer and 
the other --- by the $\mathcal M_2$-layer. 

We recall that $\dim Q = d-1$, so the caps of $Q$ are homogeneous $(d-2)$-dimensional complexes. Each cap is connected, so we need to prove that every two facets  of a cap (of dimension $d-2$) sharing a $(d-3)$-face belong to the same layer. This $(d-3)$-face is, obviously of form $\proj_I (E)$, where
$E$ is a $(d-3)$-face of $P$.

Of course, $p$ is transversal to $E$. By Lemma~\ref{lem:cases}, $E$ as a face of $T(P)$ can have only cubic or prismatic type of coincidence and, if $E$ is cubic,
$P$ has a facet $F$ or a standard $(d-2)$-face $G$ related to the dual cell $D(E)$ such that $\mathbf s(F) \in \mathcal B_p(P)$ (respectively, $\mathbf s(F) \in \mathcal A_p(P)$).

In each case $\proj_I (E)$ adjoins two facets of $Q$ covered by the same layer (either $\mathcal M_1$- or $\mathcal M_2$-). Further, if a facet of a
$\proj_I(Y_1)$-cap of $Q$ is covered by the $\mathcal M_1$-layer, its antipodal is covered by the $\mathcal M_2$-layer and vice versa. Thus the
caps are covered by different layers. This finishes the proof.

\end{proof}

\section{Sketch of the further argument}\label{sec:6}

As mentioned before, we prove Theorems~\ref{thm:voronoi_for_sum}, \ref{thm:cross} and \ref{thm:cross_reduction} simultaneously. For
Theorem~\ref{thm:voronoi_for_sum} we will use the equivalent statement of Theorem~\ref{thm:reducibility_for_twodim_free}.

We proceed by induction over $d$. At each step we will prove Theorems~\ref{thm:cross} and \ref{thm:cross_reduction} for parallelohedra of dimension 
$d-2$ and then Theorem~\ref{thm:reducibility_for_twodim_free} for $d$-dimensional parallelohedra. We emphasize that this ``dimension shift'' is important
for the argument. 

For $d\leq 4$ all the statements hold. Indeed, Theorems~\ref{thm:cross} and \ref{thm:cross_reduction} are obvious for 1- and 2-dimensional parallelohedra.
Theorem~\ref{thm:reducibility_for_twodim_free} holds for parallelohedra of dimension $d\leq 4$ because the equivalent statement of
Theorem~\ref{thm:voronoi_for_sum} is an immediate consequence of Voronoi's Conjecture. But Voronoi's Conjecture is known to be true for dimensions $\leq 4$
\cite{Del1929}. This makes the induction base.

Section~\ref{sec:7} provides a supplementary notion of {\it dilatation of Voronoi parallelohedra}. The key results here are Lemma~\ref{lem:7.2} and 
Corollary~\ref{cor:7.3} asserting that certain dilatations preserve a cross (see Definition~\ref{def:cross}). They are used in Section~\ref{sec:8}
to reduce Theorems~\ref{thm:cross} and \ref{thm:cross_reduction} to Theorem~\ref{thm:reducibility_for_twodim_free} for $(d-2)$-dimensional parallelohedra
which is true by induction hypothesis. 

Section~\ref{sec:9} reduces Theorem~\ref{thm:reducibility_for_twodim_free} for $d$-dimensional parallelohedra
to Theorems~\ref{thm:cross} and \ref{thm:cross_reduction} for $(d-2)$-dimensional parallelohedra obtained in Section~\ref{sec:8}. This completes the
induction step and the whole proof.

\section{Dilatation of Voronoi parallelohedra}\label{sec:7}

Assume that $\Lambda$ is a $d$-dimensional lattice and $\Omega$ is a positive definite quadratic form. By $P(\Lambda, \Omega)$ we will denote a 
parallelohedron, which is a Dirichlet--Voronoi cell for the lattice $\Lambda$ with respect to the Euclidean metric
$$\| \mathbf x \|_{\Omega}^2 = \mathbf x^T \Omega \mathbf x.$$

Let $\mathbf n$ be a vector. Consider a quadratic form
\begin{equation}\label{eq:7.1}
\Omega_{\mathbf n} = \Omega + \Omega^T \mathbf n \mathbf n^T \Omega.
\end{equation}
For every nonzero vector $\mathbf x$ one has
$$\mathbf x^T \Omega_{\mathbf n} \mathbf x = \mathbf x^T \Omega \mathbf x + (\mathbf n^T \Omega \mathbf x)^2 > 0,$$
thus $\Omega_{\mathbf n}$ is a positive definite quadratic form. If not otherwise stated, everywhere below we assume that $\mathbf n \neq \mathbf 0$.

\begin{defn}
All parallelohedra of form $P(\Lambda, \Omega_{\mathbf n})$ will be called {\it dilatations} of $P(\Lambda, \Omega)$.
\end{defn}

Let $\mathcal F(\Lambda, \Omega)$ be the set of all facet vectors of $P(\Lambda, \Omega)$. For what follows, we will need an another description of
facet vectors. Namely, the points $\mathbf x, \mathbf x' \in \Lambda$ are adjoint by a facet vector of $P(\Lambda, \Omega)$ iff the ball
$$B_{\Omega} (\mathbf x, \mathbf x') = 
\left \{ \mathbf y: \left \|\mathbf y - \frac{\mathbf x+\mathbf x'}{2} \right \|_{\Omega} \leq \frac 12 \|\mathbf x - \mathbf x' \|_{\Omega} \right \}$$
contains no points of $\Lambda$ other than $\mathbf x$ and $\mathbf x'$. This is because $[\mathbf x, \mathbf x']$ with $\mathbf x, \mathbf x' \in \Lambda$
is a Delaunay 1-cell iff $\mathbf x' - \mathbf x$ is a facet vector and, moreover, the empty sphere for 
the segment $[\mathbf x, \mathbf x']$ is centered at its midpoint (see~\cite[Lemma 13.2.7]{DLa1996}).

Define
$$\mathcal F_{\mathbf n}(\Lambda, \Omega) = \{\mathbf s: \mathbf s \in \mathcal F(\Lambda, \Omega)\; \text{and} \; \mathbf n^T\Omega \mathbf s \neq 0 \}.$$
The following lemma is expressed by a single formula, however, its meaning is explained in Corollary~\ref{cor:7.3}.

\begin{lem}\label{lem:7.2}
$\left\langle \mathcal F_{\mathbf n}(\Lambda, \Omega_{\mathbf n}) \right\rangle \subseteq 
\left\langle \mathcal F_{\mathbf n}(\Lambda, \Omega)\right\rangle$.
\end{lem}

\begin{proof}
Before starting the proof we emphasize an important property. For every vector $\mathbf x$ and every real $\lambda$ the conditions
$$\mathbf n^T \Omega \mathbf x = 0 \quad \text{and} \quad \mathbf n^T \Omega_{\lambda \mathbf n} \mathbf x = 0$$
are equivalent. This is an immediate consequence of the formula~(\ref{eq:7.1}).

Consider the Delaunay tiling with vertex set $\Lambda$ in the Euclidean metric given by a quadratic form $\Omega_{\lambda \mathbf n}$. We will observe
the change of the set $\mathcal F_{\mathbf n}(\Lambda, \Omega_{\lambda \mathbf n})$ as $\lambda$ grows from 0 to 1.

Suppose that at some $\lambda_0 \in (0,1)$ a new vector of $\mathcal F_{\mathbf n}(\Lambda, \Omega_{\lambda \mathbf n})$ emerges. It means that
there is a pair of points $\mathbf x, \mathbf x' \in \Lambda$ with the following properties.

\begin{enumerate}
	\item For $\lambda\searrow \lambda_0$ the ball $B_{\Omega_{\lambda \mathbf n}}(\mathbf x, \mathbf x')$ contains no points of $\Lambda$ other than 
	$\mathbf x$	and $\mathbf x'$.
	\item For $\lambda\nearrow \lambda_0$ the ball $B_{\Omega_{\lambda \mathbf n}}(\mathbf x, \mathbf x')$ contains some other points of $\Lambda$.
	\item $\mathbf n^T \Omega (\mathbf x' - \mathbf x) \neq 0$.
\end{enumerate}

If we prove that for a sufficiently small $\varepsilon>0$ the inclusion
$$\mathbf x' - \mathbf x \in \mathcal F_{\mathbf n}(\Lambda, \Omega_{(\lambda_0 - \varepsilon) \mathbf n})$$
holds, then we are done. Indeed, the inclusion means that $\mathcal F_{\mathbf n}(\Lambda, \Omega_{\lambda \mathbf n})$ never expands as 
$\lambda$ grows from 0 to 1.

Consider the ball $B_{\Omega_{\lambda_0 \mathbf n}}(\mathbf x, \mathbf x')$. By continuity, it contains some points of $\Lambda$ distinct from
$\mathbf x$	and $\mathbf x'$, but only on the boundary. Thus
$$D = \conv (B_{\Omega_{\lambda_0 \mathbf n}}(\mathbf x, \mathbf x') \cap \Lambda)$$
is a centrally symmetric Delaunay cell for the metric $\| \cdot \|_{\Omega_{\lambda_0 \mathbf n}}$ of dimension at least 2. 
It is not hard to see that all edges of $D$ are also Delaunay edges for every metric $\| \cdot \|_{\Omega_{(\lambda_0 - \varepsilon) \mathbf n}}$ 
if $\varepsilon$ is positive and small enough. 

We say that a point $\mathbf y \in \Lambda$ is above (below, on the same level with) a point $\mathbf y' \in \Lambda$ if
$\mathbf n \Omega (\mathbf y - \mathbf y')$ is positive (negative, zero respectively). As $\mathbf x$	and $\mathbf x'$ are not on the same level,
we will assume that $\mathbf x'$ is above $\mathbf x$.

We aim to prove that $\mathbf x$	and $\mathbf x'$ can be adjoint by a sequence of edges of $D$ in such a way that every edge of a sequence goes
between two vertices of different levels. This will imply that $\mathbf x' - \mathbf x$ is a combination of facet vectors of 
$P(\Lambda, \Omega_{(\lambda_0 - \varepsilon) \mathbf n})$, and Lemma~\ref{lem:7.2} will proved.

Observe that a vertex of $D$ is inside $B_{\Omega_{(\lambda_0 - \varepsilon) \mathbf n}}$ if it is above $\mathbf x'$ or below $\mathbf x$.
Since $D$ has a center of symmetry at $\tfrac{\mathbf x + \mathbf x'}{2}$, $D$ has vertices both above $\mathbf x'$ and below $\mathbf x$.

Further, $D$ has no point $\mathbf z \neq \mathbf x$ on the same level with $\mathbf x$. Indeed, assume the converse. Then the points
$\mathbf x$, $\mathbf x'$, $\mathbf z$ and $\mathbf z' = \mathbf x + \mathbf x' - \mathbf z$ lie on the sphere $B_{\Omega_{\lambda_0 \mathbf n}}$.
Thus
\begin{multline}\label{eq:7.2}
\left \| \mathbf x - \frac{\mathbf x + \mathbf x'}2 \right\|_{\Omega_{\lambda_0 \mathbf n}} = 
\left \| \mathbf x' - \frac{\mathbf x + \mathbf x'}2 \right\|_{\Omega_{\lambda_0 \mathbf n}} = 
\left \| \mathbf z - \frac{\mathbf x + \mathbf x'}2 \right\|_{\Omega_{\lambda_0 \mathbf n}} = \\
\left \| \mathbf z' - \frac{\mathbf x + \mathbf x'}2 \right\|_{\Omega_{\lambda_0 \mathbf n}}.
\end{multline}

Since $\mathbf z$ is on the same level with $\mathbf x$ and $\mathbf z'$ is on the same level with $\mathbf x'$, it is clear that
\begin{multline*}
\mathbf n^T \Omega \left(\mathbf x - \frac{\mathbf x + \mathbf x'}2 \right) = 
\mathbf n^T \Omega \left(\mathbf z - \frac{\mathbf x + \mathbf x'}2 \right) = 
- \mathbf n^T \Omega \left(\mathbf x' - \frac{\mathbf x + \mathbf x'}2 \right) = \\
- \mathbf n^T \Omega \left(\mathbf z' - \frac{\mathbf x + \mathbf x'}2 \right) 
\end{multline*}

Therefore (\ref{eq:7.2}) holds after substituting all instances of $\lambda_0 \mathbf n$ with $\lambda \mathbf n$ for every real $\lambda$. As a result,
$[\mathbf x, \mathbf x']$ is never a Delaunay edge, because the empty sphere centered at its midpoint necessarily contains at least two more points.

A well-known fact from linear programming~\cite[\S~3.2]{Zie1993} tells that $\mathbf x'$ can be connected with at least one of the highest vertices
(call this vertex $\mathbf y$) of $D$ by a sequence of edges going strictly upwards. We have proved that $\mathbf x'$ cannot be the highest point of $D$, 
so $\mathbf y \neq \mathbf x'$.

Consider the segment $[\mathbf x, \mathbf y]$. If it is an edge of $D$, the proof is over. Otherwise (see~\cite[Lemma 13.2.7]{DLa1996} again)
$[\mathbf x, \mathbf y]$ is a diagonal of a centrally symmetric face $D' \subset D$.

Suppose that $\mathbf y$ is not the only highest point of $D'$. Then there is a vertex $\mathbf y' \in D'$ on the same level with $\mathbf y$. But
due to the central symmetry, the point $\mathbf z = \mathbf x + \mathbf y - \mathbf y'$ is a vertex of $D'$ and is on the same level with $\mathbf x$.
This is impossible, so $\mathbf y$ is the only highest point of $D'$.

Thus one can connect $\mathbf x$ and $\mathbf y$ by a sequence of edges going strictly upwards. As a result, we have connected $\mathbf x$ and
$\mathbf x'$ by a sequence of edges going first strictly upwards and then strictly downwards. Thereby we have completed the remaining part of the proof.

\end{proof}

\begin{cor}\label{cor:7.3}
Assume that the parallelohedron $P(\Lambda, \Omega)$ has a cross of hyperplanes $\Pi, \Pi'$ (by Definition~\ref{def:cross}, it means that every facet vector of $P(\Lambda, \Omega)$ is parallel to $\Pi$ or to $\Pi'$). Let $\mathbf n$ be a normal vector to $\Pi$ in the metric $\|\cdot\|_{\Omega}$. Then 
$P(\Lambda, \Omega_{\mathbf n})$ has the same cross $(\Pi, \Pi')$.
\end{cor}

\begin{proof}
The property of $P(\Lambda, \Omega)$ to have a cross $\Pi, \Pi'$ means that 
$$\left\langle \mathcal F_{\mathbf n}(\Lambda, \Omega) \right\rangle \subseteq \Pi'.$$

By Lemma~\ref{lem:7.2}, 
$$\left\langle \mathcal F_{\mathbf n}(\Lambda, \Omega_{\mathbf n})\right\rangle \subseteq 
\left\langle \mathcal F_{\mathbf n}(\Lambda, \Omega) \right\rangle \subseteq \Pi' .$$

But the set $\mathcal F(\Lambda, \Omega_{\mathbf n})\setminus \mathcal F_{\mathbf n}(\Lambda, \Omega_{\mathbf n})$ lies in the orthogonal complement to 
$\mathbf n$ (which is the same in both $\| \cdot \|_{\Omega}$ and $\| \cdot \|_{\Omega_{\mathbf n}}$), i.e. in $\Pi$. Thus
$$\mathcal F(\Lambda, \Omega_{\mathbf n}) \subset \Pi \cup \Pi',$$
which means that $P(\Lambda, \Omega_{\mathbf n})$ has the cross $(\Pi, \Pi')$.

\end{proof}

\section{Induction step for Theorems \ref{thm:cross} and \ref{thm:cross_reduction}}\label{sec:8}

The goal of this section is to prove Theorems~\ref{thm:cross}~and~\ref{thm:cross_reduction} for parallelohedra of dimension $n = d-2$ provided that
they are proved for smaller dimensions and Theorem~\ref{thm:reducibility_for_twodim_free} is proved for dimension $n$. (In fact, the induction
hypothesis asserts that Theorem~\ref{thm:reducibility_for_twodim_free} is true for all dimensions up to $n+1$.) The proof is given as a series of lemmas.

\begin{lem}\label{lem:8.1}
Theorem~\ref{thm:cross_reduction} is true for $n$-dimensional parallelohedra if Theorem~\ref{thm:cross} is true for parallelohedra of all
dimensions up to $n-1$.
\end{lem}

\begin{proof}
Let $\dim P = n$ and $P = P_1 \oplus P_2 \oplus \ldots \oplus P_k$, where $k > 1$ and all $P_i$ are irreducible and let 
$(\Pi_1, \Pi_2)$ be a cross for $P$. We have to prove that $\aff P_i \parallel \Pi_1$ or $\aff P_i \parallel \Pi_2$ for each $i=1,2,\ldots,k$.

Assume the converse, say, $\aff P_1 \nparallel \Pi_1$ and $\aff P_i \nparallel \Pi_2$. Then 
$$ (\lin \aff P_1 \cap \Pi_1,\; \lin \aff P_1 \cap \Pi_2) $$
is a pair of hyperplanes in $\lin \aff P_1$ being a cross for $P_1$. But $\dim P_1 < n$, therefore by Theorem~\ref{thm:cross} the parallelohedron
$P_1$ is reducible, a contradiction.

\end{proof}

\begin{lem}\label{lem:8.2}
Assume that a Voronoi $n$-parallelohedron $P(\Lambda, \Omega)$ has a cross $(\Pi_1, \Pi_2)$ and the lattices
$$\Lambda \cap \Pi_1 \quad \text{and} \quad \Lambda \cap \Pi_2$$
are $(n-1)$-dimensional.  Then there are vectors $\mathbf n_1$ and $\mathbf n_2$ such that
\begin{enumerate}
	\item[\rm 1.] $\mathbf n_1$ is orthogonal to $\Pi_1$ in $\|\cdot \|_{\Omega}$.
	\item[\rm 2.] $\mathbf n_2$ is orthogonal to $\Pi_2$ in $\|\cdot \|_{\Omega_{\mathbf n_1}}$.
	\item[\rm 3.] The twofold dilatation $P(\Lambda, (\Omega_{\mathbf n_1})_{\mathbf n_2})$ has a free space 
	$\left\langle \mathbf n_1, \mathbf n_2 \right\rangle$.
\end{enumerate}
\end{lem}

\begin{proof}
The lattice $\Lambda \cap \Pi_1 \cap \Pi_2$
is $(n-2)$-dimensional. Indeed, $\Pi_1$ and $\Pi_2$ have bases consisting of integer vectors, so they can be restricted to hyperplanes in $\mathbb Q^n$.
Therefore $\Pi_1 \cap \Pi_2$ restricted to $\mathbb Q^n$ is a $(d-2)$-dimensional linear space. Hence it has a rational basis and, under a proper scaling,
an integer basis.

For every possible choice of $(\Omega_{\mathbf n_1})_{\mathbf n_2}$ its restriction to $\Pi_1 \cap \Pi_2$ coincides with the
restriction of $\Omega$ to the same space. Denote this rectriction by $\Omega'$. Let $\rho$ be the radius of the largest empty sphere for the lattice 
$\Lambda \cap \Pi_1 \cap \Pi_2$ with respect to the metric $\|\cdot \|_{\Omega'}$.

Let $\mathbf n'_1$ be an arbitrary normal to $\Pi_1$ in $\|\cdot \|_{\Omega}$. Then the whole lattice $\Lambda$ can be covered by a bundle $\mathcal X_1$
of hyperplanes
$$ (\mathbf n'_1)^T \Omega \mathbf x = m\alpha, \quad m\in \mathbb Z.$$
Set $\mathbf n_1 = \left| \tfrac{\rho}{\alpha} \right| \mathbf n'_1$. Then the distance between the hyperplanes
$$ (\mathbf n'_1)^T \Omega \mathbf x = m\alpha \quad \text{and} \quad (\mathbf n'_1)^T \Omega \mathbf x = (m+1)\alpha $$
in $\|\cdot \|_{\Omega_{\mathbf n_1}}$ equals $|\alpha|\sqrt{1+\tfrac{\rho^2}{\alpha^2}} > \rho$.

Similarly, let $\mathbf n'_2$ be an arbitrary normal to $\Pi_2$ in $\|\cdot \|_{\Omega_{\mathbf n_1}}$. Then the whole lattice $\Lambda$ can be 
covered by a bundle $\mathcal X_2$ of hyperplanes
$$ (\mathbf n'_2)^T \Omega_{\mathbf n_1} \mathbf x = m\beta, \quad m\in \mathbb Z.$$
Set $\mathbf n_2 = \left| \tfrac{\rho}{\beta} \right| \mathbf n'_2$. Then the distance between the hyperplanes
$$ (\mathbf n'_2)^T \Omega_{\mathbf n_1} \mathbf x = m\beta \quad \text{and} \quad (\mathbf n'_2)^T \Omega_{\mathbf n_1} \mathbf x = (m+1)\beta $$
in $\|\cdot \|_{(\Omega_{\mathbf n_1})_{\mathbf n_2}}$ equals $|\beta|\sqrt{1+\tfrac{\rho^2}{\beta^2}} > \rho$.

Changing the metric from $\|\cdot \|_{\Omega_{\mathbf n_1}}$ to $\|\cdot \|_{(\Omega_{\mathbf n_1})_{\mathbf n_2}}$ does not decrease the distances,
so the $(\Omega_{\mathbf n_1})_{\mathbf n_2}$-distance between two consecutive planes of $\mathcal X_1$ is still greater than $\rho$.

Consider the Delaunay tiling $\mathcal D$ for lattice $\Lambda$ and metric $\| \cdot \|_{(\Omega_{\mathbf n_1})_{\mathbf n_2}}$. We prove that every
triangle $\Delta\in \mathcal D$ has an edge parallel to $\Pi_1 \cap \Pi_2$.

By Corollary~\ref{cor:7.3}, every edge of $\mathcal D$ is parallel to $\Pi_1$ or $\Pi_2$. By Pigeonhole principle, $\Delta$ has two edges parallel
to the same hyperplane, say, $\Pi_2$. Then $\aff \Delta \parallel \Pi_2$.

Assume that no edge of $\Delta$ is parallel to $\Pi_1 \cap \Pi_2$. Then no edge of $\Delta$ is parallel to $\Pi_1$. Then the vertices
of $\Delta$ belong to pairwise different planes of $\mathcal X_1$. Denote the vertices of
$\Delta$ by $\mathbf x_1$, $\mathbf x_2$ and $\mathbf x_3$. Without loss of generality assume that the plane of $\mathcal X_1$ passing through $\mathbf x_2$
lies between the planes of $\mathcal X_1$ passing through $\mathbf x_1$ and $\mathbf x_3$.

Consider a subbundle $\mathcal X'_1 \subset \mathcal X_1$ consisting of those hyperplanes of $\mathcal X_1$ that have at least one integer point in common with
$\aff \Delta$. Of course, the hyperplanes of $\mathcal X'_1$ are equally spaced, and the intersection of each with $\Pi_2$ contains a $(d-2)$-lattice.
Finally, the hyperplanes of $\mathcal X_1$ passing through $\mathbf x_1$, $\mathbf x_2$ and $\mathbf x_3$ are in $\mathcal X'_1$.

Let the interval $(\mathbf x_1, \mathbf x_3)$ be intersected by exactly $m$ hyperplanes of $\mathcal X'_1$ (obviously, $m\geq 1$). Set
$$\mathbf x_4 = \frac{m-1}{m+1} \mathbf x_1 + \frac{2}{m+1} \mathbf x_3.$$
Obviously, $\mathbf x_4 \in [\mathbf x_1, \mathbf x_3]$.

Since $[\mathbf x_1, \mathbf x_3]$ is an edge of $\mathcal D$, $\partial B_{(\Omega_{\mathbf n_1})_{\mathbf n_2}}(\mathbf x_1,\mathbf x_3)$ is
an empty sphere. Perform a homothety with center $\mathbf x_1$ and coefficient $\tfrac{2}{m+1}$. The ball 
$B_{(\Omega_{\mathbf n_1})_{\mathbf n_2}}(\mathbf x_1,\mathbf x_3)$ goes to the ball $B_{(\Omega_{\mathbf n_1})_{\mathbf n_2}}(\mathbf x_1,\mathbf x_4)$.
Since $\tfrac{2}{m+1} \leq 1$,
$$B_{(\Omega_{\mathbf n_1})_{\mathbf n_2}}(\mathbf x_1,\mathbf x_4) \subset B_{(\Omega_{\mathbf n_1})_{\mathbf n_2}}(\mathbf x_1,\mathbf x_3)$$
and therefore $\partial B_{(\Omega_{\mathbf n_1})_{\mathbf n_2}}(\mathbf x_1,\mathbf x_4)$ is an empty sphere.

By choice of $\mathbf x_4$, the point $\tfrac{\mathbf x_1 + \mathbf x_4}{2}$ lies in a plane of $\mathcal X'_1$. Thus the $(n-2)$-plane
$$\frac{\mathbf x_1 + \mathbf x_4}{2} + (\Pi_1 \cap \Pi_2)$$
contains an $(n-2)$-lattice with all empty spheres not greater than $\rho$ in radius. But the sphere
$$\partial B_{(\Omega_{\mathbf n_1})_{\mathbf n_2}}(\mathbf x_1,\mathbf x_4) \cap \left( \frac{\mathbf x_1 + \mathbf x_4}{2} + (\Pi_1 \cap \Pi_2) \right)$$
is empty and has radius 
$$\frac 12 \| \mathbf x_4 - \mathbf x_1 \|_{(\Omega_{\mathbf n_1})_{\mathbf n_2}} > \rho,$$
because $\mathbf x_1$ and $\mathbf x_4$ belong to two non-consecutive planes of $\mathcal X_1$. A contradiction, thus every
triangle of $\mathcal D$ has an edge parallel to $\Pi_1 \cap \Pi_2$.

Hence, by Theorem~\ref{thm:1.2}, the orthogonal complement to $\Pi_1 \cap \Pi_2$ in $\| \cdot \|_{(\Omega_{\mathbf n_1})_{\mathbf n_2}}$
is a free space for $P(\Lambda, (\Omega_{\mathbf n_1})_{\mathbf n_2})$. It is not hard to check that $\mathbf n_1$ and $\mathbf n_2$
are independent and both orthogonal to $\Pi_1 \cap \Pi_2$.

\end{proof}

\begin{lem}\label{lem:8.3}
Assume that Theorems~\ref{thm:reducibility_for_twodim_free}~and~\ref{thm:cross_reduction} are true for dimension $n$. Then all $n$-dimensional 
Voronoi parallelohedra with crosses are reducible.
\end{lem}

\begin{proof}
Let $P(\Lambda, \Omega)$ have a cross. Then $\mathcal F(\Lambda, \Omega)$ can be partitioned into two subsets $\mathcal F_1, \mathcal F_2$ of dimension 
less than $n$ each. If necessary, append $\mathcal F_1$ and $\mathcal F_2$ by several vectors of $\Lambda$ to obtain generating sets of two hyperplanes 
$\Pi_1$ and $\Pi_2$ respectively. By construction, $(\Pi_1, \Pi_2)$ is a cross for $P(\Lambda, \Omega)$ satisfying the conditions of Lemma~\ref{lem:8.2}.

Consider the parallelohedron $P(\Lambda, (\Omega_{\mathbf n_1})_{\mathbf n_2})$ introduced in Lemma~\ref{lem:8.2}. It has a two-dimensional free
space $\left\langle \mathbf n_1, \mathbf n_2 \right\rangle$. In addition, by Corollary~\ref{cor:7.3}, $(\Pi_1, \Pi_2)$ is a cross for  
$P(\Lambda, (\Omega_{\mathbf n_1})_{\mathbf n_2})$ as well.

By Theorem~\ref{thm:reducibility_for_twodim_free} for dimension $n$,
$$P(\Lambda, (\Omega_{\mathbf n_1})_{\mathbf n_2}) = P_1 \oplus P_2 \oplus \ldots \oplus P_k.$$
In turn, Theorem~\ref{thm:cross_reduction} says that $\aff P_j \parallel \Pi_1$ or $\aff P_j \parallel \Pi_2$.

Let $R_1$ be the sum of all summands that are parallel to $\Pi_1$ and $R_2$ be the sum of the remaining summands. Then
$$P(\Lambda, (\Omega_{\mathbf n_1})_{\mathbf n_2}) = R_1 \oplus R_2,$$
where $\aff R_j \parallel \Pi_j$ ($j = 1,2$). Obviously, $\aff R_1$ and $\aff R_2$ are orthogonal with respect to $(\Omega_{\mathbf n_1})_{\mathbf n_2}$.

Thus $(\Omega_{\mathbf n_1})_{\mathbf n_2} = \Omega_1 + \Omega_2$, where $\Omega_1$ and $\Omega_2$ are positive semidefinite quadratic forms with
kernels $\lin \aff R_2$ and $\lin \aff R_1$ respectively.

The kernel of $(\Omega_{\mathbf n_1})_{\mathbf n_2} - \Omega_{\mathbf n_1}$ contains $\lin \aff R_2$. Thus the kernel of
$$\Omega'_1 = \Omega_1 - (\Omega_{\mathbf n_1})_{\mathbf n_2} + \Omega_{\mathbf n_1}$$
contains $\lin \aff R_2$. But the form $\Omega'_1$ is positive definite on $\lin \aff R_1$, otherwise the form $\Omega_{\mathbf n_1} = \Omega_2 + \Omega'_1$
is not positive definite.

$\Omega'_1$ and $\Omega_2$ have complementary kernels $\lin \aff R_2$ and $\lin \aff R_1$ respectively, therefore 
$P(\Lambda, \Omega_{\mathbf n_1})= R'_1 \oplus R_2$, where $\aff R'_1 \parallel \Pi_1$. Thus, in addition, $(\Pi_1, \Pi_2)$ is a cross for $P(\Lambda, \Omega_{\mathbf n_1})$.

Repeating the same argument for $P(\Lambda, \Omega_{\mathbf n_1})$ we obtain that $P(\Lambda, \Omega)$ is reducible and has the cross $(\Pi_1, \Pi_2)$.

\end{proof}

\section{Voronoi parallelohedra with free planes are reducible}\label{sec:9}

In this section we complete the proof of our main results by explaining the induction step in Theorem~\ref{thm:reducibility_for_twodim_free}.
This requires Theorems~\ref{thm:cross}~and~\ref{thm:cross_reduction} for dimension $d-2$. These statements were enabled for being used in Section~\ref{sec:8}.

We will use the results of Section~\ref{sec:5} extensively. In order to do this, we prove the following.

\begin{lem}\label{lem:9.1}
If a parallelohedron has a free two-dimensional plane, then it has a free perfect two-dimensional plane.
\end{lem}

\begin{proof}
Assume that $P$ be a parallelohedron and $p$ is a free plane for $P$. Let $\pm F_1, \pm F_2, \ldots, \pm F_k$ be all the facets of $P$ parallel to $p$.
Obviously, each six-belt of $P$ contains at least one pair $\pm F_j$, otherwise $p$ is not free. Further,
$$\dim \left(\lin\aff F_1 \cap \lin\aff F_2 \cap \ldots \cap \lin\aff F_k \right) \geq 2,$$
as the intersection contains $p$.

If necessary, add facets $\pm F_{k+1}, \pm F_{k+2}, \ldots \pm F_m$ so that 
$$\dim p' = 2, \quad \text{where} \quad p'= \left(\lin\aff F_1 \cap \lin\aff F_2 \cap \ldots \cap \lin\aff F_m \right).$$
Then the conditions of Theorem~\ref{thm:1.2} hold for every segment $I \parallel p'$. Thus $p'$ is a free plane for $P$.

\end{proof}

The remaining part of the proof is presented as a series of lemmas.

\begin{lem}\label{lem:9.2}
Let $R = P(\Lambda, \Omega)$ be a Voronoi parallelohedron. Assume in addition that $R$ is centered at the origin and $\mathbf 0 \in \Lambda$.

Let $\mathbf v$ be a vector. Call a facet $F\subset R$ good, if the point $\mathbf v + \tfrac 12 \mathbf s(F)$, which is the center of the
facet $F+\mathbf v \subset R+\mathbf v$, is disjoint from all facets of $T(R)$ parallel to $F$. Otherwise call $F$ bad.

Finally, let 
$$\mathbf v' \in (\Lambda + \mathbf v) \cap R.$$ 
Then the vector $\mathbf v'$ is parallel to all bad facets of $R$. 
\end{lem}

\begin{proof}
Let $F \subset R$ be a bad facet. Then, by definition of a bad facet, the point $\mathbf v + \tfrac 12 \mathbf s(F)$ belongs to some facet
$F + \mathbf t$, where $\mathbf t \in \Lambda$. It means that the polytopes $F + \mathbf t$ and $F + \mathbf v$ have a common point
$\mathbf v + \tfrac 12 \mathbf s(F)$.

Therefore the polytopes $F$ and $F + \mathbf v - \mathbf t$ share a common point 
$$\mathbf v - \mathbf t + \frac 12 \mathbf s(F).$$

Hence (see~\cite{DGr1962} for details), 
\begin{equation}\label{eq:9.1}
	\mathbf v - \mathbf t \in \frac 12 F + \frac 12 (-F).
\end{equation}

The inclusion (\ref{eq:9.1}) has two immediate consequences. First of all, $\mathbf v - \mathbf t \parallel F$. Secondly, since $-F$ is also a face of $R$,
$\mathbf v - \mathbf t \in R$. Thus we have found a particular vector from $(\Lambda + \mathbf v) \cap R$ which is parallel to $F$. Now we
have to prove the same parallelity for all other vectors of $(\Lambda + \mathbf v) \cap R$.

$R$ is a fundamental domain for the translation group $\Lambda$. Consequently, if $\mathbf v - \mathbf t \in \inter R$, then $(\Lambda + \mathbf v) \cap R$
consists of the only vector $\mathbf v - \mathbf t$, which is parallel to $F$, as proved above.

Now suppose that $\mathbf v - \mathbf t \in \partial R$.
Let $E$ be the minimal face of $R$ containing the point $\mathbf v - \mathbf t$. All the elements of $(\Lambda + \mathbf v) \cap R$ are representable as
$\mathbf v - \mathbf t + \mathbf t'$, where $\mathbf t' \in \Lambda$ and $E+ \mathbf t' \subset R$. For a Voronoi parallelohedron $R$ it is well-known
that $E \subset R$ and $E+\mathbf t' \subset R$ together give $\mathbf t' \, \bot \, E$ with orthogonality related to $\| \cdot \|_{\Omega}$.

On the other hand, $\tfrac 12 F + \tfrac 12 (-F)$ is a mid-section of the prism $\conv (F\cup (-F))$. Therefore (\ref{eq:9.1}) guarantees that 
if $\mathbf v - \mathbf t \in E$, then necessarily 
$$\left [\mathbf v - \mathbf t - \frac 12 \mathbf s(F), \mathbf v - \mathbf t + \frac 12 \mathbf s(F) \right] \subset E.$$

Hence $\mathbf s(F) \in \lin \aff E$ and, consequently, $\mathbf t' \, \bot \, \mathbf s(F)$. As a result, $\mathbf t' \parallel F$, and finally,
$\mathbf v - \mathbf t + \mathbf t' \parallel F$.

\end{proof}

\begin{lem}\label{lem:9.3}
Let a Voronoi parallelohedron $P$ have a free perfect two-dimensional plane $p$. Then $P$ is a prism, or the parallelohedron $R = \proj_p (P)$ has a cross.
\end{lem}

\begin{proof}
Recall that $p$ contains two perfect free lines $\ell_1$ and $\ell_2$ and let the segment $I$ be parallel to $p$, but non-parallel to both $\ell_j$.
Again, let the segments $Y_j$ to be parallel to $\ell_j$. In Section~\ref{sec:5} we have defined the sets $\mathcal C^j_I(P)$ for $j = 1,2$. 
As in Lemma~\ref{lem:5.6}, let 
$$\mathbf w_j \in \mathcal C^j_I(P), \quad \text{and}$$
$$\Lambda_j = \Lambda(P) \cap \left( \mathbf w_j +  \left\langle \mathcal A_p(P) \cup \mathcal B_p(P) \right\rangle \right).$$

By Lemma~\ref{lem:5.3}, $P+Y_1+Y_2$ is a parallelohedron. Since it has a nonzero width in the direction $p$, the sets 
$$T_j = \{ \proj_p (P + \mathbf t) : \mathbf t \in \Lambda_j \} \quad (j=1,2)$$
both are tilings of $\mathbb R^{d-2}$ by translates of a parallelohedron $R = \proj_p(P)$. Choose $\mathbf v_1$ and $\mathbf v_2$ so that
$$R - \mathbf v_j \in T_j.$$

Let $\mathbf v_1 \in \Lambda (R)$. Then $R$ is a tile of $\mathbf T_1$, and it is the only tile of $T_1$ to have a $(d-2)$-dimensional intersection with $R$.

From Lemma~\ref{lem:5.6} it immediately follows that
$$|\mathcal C^j_I(P)| = 1.$$
Thus all but two facets of $P$ are parallel to $\ell_2$. As an immediate consequence we get that $P$ is a prism. Similarly, $P$ is a prism if 
$\mathbf v_2 \in \Lambda (R)$.

By Lemma~\ref{lem:5.4}, $R = P(\Lambda(R), \Omega)$ for some positive quadratic form $\Omega$ of $(d-2)$ variables. Now in terms of Lemma~\ref{lem:9.2}, 
assume that every facet of $R$ is good with respect at least to one vector $\mathbf v_1$ or $\mathbf v_2$. Choose
$$\mathbf v'_j \in R \cap (\mathbf v_j + \Lambda(R)).$$
The cases $\mathbf v'_j = \mathbf 0$ have been considered before, so assume that $\mathbf v'_j \neq \mathbf 0$ for $j=1,2$.

Then, according to Lemma~\ref{lem:9.2}, each facet of $R$ is parallel to $\mathbf v'_1$ or $\mathbf v'_2$. Equivalently, each facet vector of $R$
is orthogonal to $\mathbf v'_1$ or $\mathbf v'_2$ in the metric $\| \cdot \|_{\Omega}$. Thus orthogonal complements to $\mathbf v'_1$ and $\mathbf v'_2$
form a cross for $R$.

We will prove that nothing else is possible. Namely, no facet of $R$ can be bad with respect both to $\mathbf v_1$ and $\mathbf v_2$.

Assume that $E'$ is a facet of $R$ that is bad with respect to $\mathbf v_1$ and $\mathbf v_2$. Then, obviously there exist
$R_1 \in T_1$ and $R_2 \in T_2$ satisfying
\begin{equation}\label{eq:9.2}
E' \cap \inter R_1 \cap \inter R_2 \neq \varnothing.
\end{equation}
Indeed, in the sense of $(d-3)$-dimensional Lebesgue measure, almost every point of $E'$ close enough to its center is covered by exactly one
tile of $T_1$ and exactly one tile of $T_2$.

Let $R_j = \proj_p (P_j)$, where $P_j =  P + \mathbf t_j$, $\mathbf t_j \in \Lambda_j$ and $j = 1,2$. Then, by Lemma~\ref{lem:5.6}, 
the face $P \cap P_1 \cap P_2$ is $(d-2)$-dimensional and has a $(d-3)$-subface $E$ such that 
$$\proj_p (E) = E' \cap R_1 \cap  R_2.$$

Since 
$$\dim \aff E = \dim \aff \proj_p(E) = d-3,$$
the plane $p$ is transversal to $E$. This corresponds to one of the cases of Lemma~\ref{lem:cases}, item 2.
But none of these cases matches with (\ref{eq:9.2}), a contradiction. Hence $R$ cannot have facets which are bad with respect both to $\mathbf v_1$ 
and $\mathbf v_2$.

\end{proof}

\begin{lem}\label{lem:9.4}
Let a Voronoi $d$-parallelohedron $P$ have a free two-dimensional plane $p$. Assume that Theorems~\ref{thm:cross}~and~\ref{thm:cross_reduction} 
hold for dimension $n = d-2$. Then $P$ is reducible.
\end{lem}

\begin{proof}
Lemma~\ref{lem:9.1} asserts that $P$ has a perfect free plane. Therefore let $p$ be perfect for the rest of the proof.

We will use the notation of Lemma~\ref{lem:9.3}. We also assume that the image space of $\proj_p$ is 
$\left\langle \mathcal A_p(P) \cup \mathcal B_p(P) \right\rangle$.

If $P$ is a prism, then, obviously, the assertion of Lemma~\ref{lem:9.4} is true.
By Lemma~\ref{lem:9.3}, if $P$ is not a prism, then every facet vector of the parallelohedron $R = \proj_p(P)$ is orthogonal to at least
one of the two vectors $\mathbf v'_1$ and $\mathbf v'_2$. Thus $R$ is reducible, and Theorem~\ref{thm:cross_reduction} gives that
$$R = S_1 \oplus S_2, \quad \text{where} \quad \aff S_j \, \bot \, \mathbf v'_j, \quad j = 1,2.$$

Hence $\mathbf v'_1 \in \lin \aff S_2$ and $\mathbf v'_2 \in \lin \aff S_1$. Consequently,
if $\mathbf t \in \Lambda_1$ and $\dim \aff ((R+\mathbf t)\cap R) = d-2$ (respectively, $\mathbf t \in \Lambda_2$ and $\dim \aff ((R+\mathbf t)\cap R) = d-2$),
then 
$$\proj_p (\mathbf t) \in \lin \aff S_2 \quad \text{(respectively, $\proj_p (\mathbf t) \in \lin \aff S_1$)}.$$

But if $F$ is a facet of $P$ and $\mathbf s(F) \in \mathcal C_I(P)$, then
$$\proj_p \bigl(P\cap (P+\mathbf s(F))\bigr) = R \cap (R+\mathbf t),$$
where $\mathbf t$ denotes $\mathbf s(F)$. In particular, this gives
$$\dim \aff \bigl(R \cap (R+\mathbf t)\bigr) = d - 2.$$

As a result,
\begin{equation}\label{eq:9.3}
	\mathcal C^1_I(P) \subset \mathbf w_1 + \lin \aff S_2, \qquad \mathcal C^2_I(P) \subset \mathbf w_2 + \lin \aff S_1.
\end{equation}

Further, every vector of $\mathcal B_p(P)$ corresponds to a facet vector of $R$, so 
\begin{equation}\label{eq:9.4}
	\mathcal B_p(P) \in \lin \aff S_1 \cup \lin \aff S_2.
\end{equation}

Combining (\ref{eq:9.3}) and (\ref{eq:9.4}), we obtain that every facet vector of $P$ belongs to one of the two complementary spaces
$$\left\langle \mathbf w_1 \right\rangle \oplus \lin \aff S_2 \quad \text{and} \quad \left\langle \mathbf w_2 \right\rangle \oplus \lin \aff S_1.$$
By Theorem~\ref{thm:ordine_reducibility_misc}, $P$ is reducible.

\end{proof}

The proof of Lemma~\ref{lem:9.4} finishes the whole induction step.

\section*{Acknowledgements}
The author is supported by the Russian government project 11.G34.31.0053 and RFBR grant 11-01-00633. He would like to thank N.~Dolbilin for scientific
guidance, V.~Grishukhin for proposing the problem and A.~Garber for bringing the problem into the discussion.

A major part of the work has been done in Kingston, Canada during several visits to prof.~R.~Erdahl, who contributed a lot to the discussion of rank one
quadratic forms and dual cells. The author is also grateful to A.~Gavrilyuk, \'A.~Horv\'ath and A.~V\'egh for comments and suggestions.

\end{document}